\documentclass[11pt,reqno]{amsart}
\usepackage[letterpaper,margin=1in,footskip=1in]{geometry}
\usepackage{microtype}
\usepackage[only,llbracket,rrbracket]{stmaryrd}
\usepackage{enumitem}
\usepackage{amsmath}
\usepackage{soul}
\usepackage{amssymb}
\usepackage{mathtools}
\usepackage{tikz-cd}
\usepackage{mathrsfs}
\usepackage{verbatim}
\usepackage[title]{appendix}
\usepackage{newunicodechar}
\newunicodechar{ב}{\beth}
\usepackage{subfig}
\usepackage{xcolor}
\usepackage{comment}
\usepackage[justification=raggedright]{caption}
\usepackage[noadjust]{cite}

\usepackage{combelow}
\usepackage{kantlipsum}
\usepackage{bm}

\allowdisplaybreaks
\PassOptionsToPackage{pdfusetitle,colorlinks,bookmarksdepth=3}{hyperref}
\usepackage{bookmark}
\usepackage[T1]{fontenc}
\hypersetup{citecolor=[HTML]{2E7E2A},linkcolor=[HTML]{2E7E2A},urlcolor=[HTML]{2E7E2A}}
\newtheorem{theorem}{Theorem}[section]
\newtheorem{lemma}[theorem]{Lemma}
\newtheorem{cor}[theorem]{Corollary}
\newtheorem{proposition}[theorem]{Proposition}

\newtheorem{alphthm}{Theorem}

\newcommand{\mylabel}[2]{#2\def\@currentlabel{#2}\label{#1}}
\theoremstyle{definition}
\newtheorem{definition}[theorem]{Definition}
\newtheorem{example}[theorem]{Example}

\newtheorem{construction}[theorem]{Construction}
\newtheorem{convention}[theorem]{Convention}
\theoremstyle{remark}
\newtheorem{remark}[theorem]{Remark}

\newtheoremstyle{citedrem}{.5\baselineskip\@plus.2\baselineskip\@minus.2\baselineskip}{.5\baselineskip\@plus.2\baselineskip\@minus.2\baselineskip}{}{}{\itshape}{\itshape .}{5pt plus 1pt minus 1pt}{\thmname{#1}\thmnumber{ \normalfont#2}\thmnote{ \normalfont#3}}
\usepackage[font=small,labelfont=bf]{caption}
\theoremstyle{citedrem}

\newtheoremstyle{cited}{.5\baselineskip\@plus.2\baselineskip\@minus.2\baselineskip}{.5\baselineskip\@plus.2\baselineskip\@minus.2\baselineskip}{\itshape}{}{\bfseries}{\bfseries .}{5pt plus 1pt minus 1pt}{\thmname{#1}\thmnumber{ #2}\thmnote{ \normalfont#3}}
\theoremstyle{cited}

\newtheoremstyle{citeddef}{.5\baselineskip\@plus.2\baselineskip\@minus.2\baselineskip}{.5\baselineskip\@plus.2\baselineskip\@minus.2\baselineskip}{}{}{\bfseries}{\bfseries .}{5pt plus 1pt minus 1pt}{\thmname{#1}\thmnumber{ #2}\thmnote{ \normalfont#3}}
\theoremstyle{citeddef}

\newtheoremstyle{step}{.25\baselineskip\@plus.1\baselineskip\@minus.1\baselineskip}{.25\baselineskip\@plus.1\baselineskip\@minus.1\baselineskip}{\itshape}{}{\bfseries}{\bfseries .}{5pt plus 1pt minus 1pt}{\thmname{#1}\thmnumber{ #2}\thmnote{ \normalfont(#3)}}
\theoremstyle{step}

\newcommand{\N}{\mathbf{N}}

\newcommand{\R}{\mathbf{R}}
\newcommand{\C}{\mathbf{C}}

\newcommand{\spec}{\mathrm{Spec}}

\newcommand{\na}{\mathrm{na}}

\newcommand{\val}{\mathrm{val}}

\newcommand{\trop}{\mathrm{trop}}
\newcommand{\hyb}{\mathrm{hyb}}
\newcommand{\sk}{\mathrm{Sk}}
\newcommand{\FS}{\mathrm{FS}}
\newcommand{\TY}{\mathrm{TY}}
\newcommand{\CL}{\mathrm{CL}}
\newcommand{\CY}{\mathrm{CY}}

\newcommand{\MA}{\mathrm{MA}}

\newcommand{\RMA}{\mathrm{MA}_{\mathbf{R}}}
\newcommand{\CMA}{\mathrm{MA}_{\mathbf{C}}}
\newcommand{\NAMA}{\mathrm{MA}^{\mathrm{na}}}
\newcommand{\Leb}{\mathrm{Leb}}
\newcommand{\QM}{\mathrm{QM}}

\newcommand{\triv}{\mathrm{triv}}
\newcommand{\NRT}{N_{\mathbf{R}}(\mathbf{T})}

\newcommand{\stepeq}{\overset{\mathrm{Step 4}}{=\joinrel=}}
\usepackage{contour}
\usepackage[normalem]{ulem}

\contourlength{0.8pt}

\DeclarePairedDelimiterX\Set[1]\{\}{#1}
\tikzcdset{scale cd/.style={every label/.append style={scale=#1},cells={nodes={scale=#1}}}}

\begin{document}
\title{Non-Archimedean Calabi--Yau Potentials on Certain Affine Varieties}

\author[Ying Wang]{Ying Wang}
\address{Department of Mathematics\\University of Michigan\\Ann Arbor, MI 48109-1043\\USA}
\email{\href{mailto:ywangx@umich.edu}{ywangx@umich.edu}}

\setlength{\abovedisplayskip}{3pt}
\setlength{\belowdisplayskip}{3pt}

\begin{abstract}
We solve a non-Archimedean Monge--Amp\`{e}re equation on the Berkovich analytification of a complex log Calabi--Yau pair whose dual complex is a standard simplex, answering a question of \cite{CL24} and offering a non-Archimedean analog of Ricci-flat metric potentials on complex affine varieties. This work builds on the solution to a complex Monge--Ampère equation obtained by \cite{CL24} and \cite{CTY24}. We also show the suitably rescaled limits of the complex potentials coincide with their non-Archimedean counterparts in some situations, strengthening their connections.
\end{abstract}

\maketitle
\setcounter{tocdepth}{1}

{
  \hypersetup{linkcolor=black}
  \tableofcontents
}

\thispagestyle{empty}

\section{Introduction}
A K\"{a}hler manifold $X$ of dimension $n$, compact or non-compact, is said to be Calabi--Yau if it admits a non-vanishing holomorphic $n$-form $\Omega$, with associated smooth measure $\mu \coloneqq (\sqrt{-1})^{n^2}\Omega \wedge \bar{\Omega}$ on $X$. If $X$ is compact, $\Omega$ and $\mu$ are unique up to scaling, and by \cite{Yau78} any ample line bundle $L$ has a unique up to scaling hermitian metric $\| \cdot \|$ whose curvature form is Ricci flat. This was done by solving the complex Monge--Amp\`{e}re equation, 
\begin{equation}\label{eq:cNA}\CMA(\| \cdot \|) \coloneqq c_1(L, \|\cdot\|)^n = \mathrm{const} \cdot  \mu.\end{equation}

For non-compact $X$, the existence of such metrics is more delicate. A first step is to study the case \[X \coloneq \bar{X} \backslash D,\] where $D$ is a simple normal crossing anticanonical divisor in a smooth projective Fano variety $\bar{X}$. In this setting, there is a natural holomorphic $n$-form $\Omega$ on $X$ with a simple pole along $D$, with associated smooth positive measure \[\mu \coloneqq (\sqrt{-1})^{n^2} \Omega \wedge \bar{\Omega}.\]

To obtain a Calabi--Yau metric on $X$, it has been fruitful to first prescribe a metric ansatz on the normal bundle $N_{D / \bar{X}}$, and then try to complete the ansatz on $X$. For irreducible $D$, the ansatz construction was pioneered by Calabi in \cite{Cal79}, nowadays known as the Calabi ansatz. Recent works \cite{CL24} and \cite{CTY24} have generalized this ansatz to allow $D$ to break into multiple components. These ansatzs all arise from solving a complex Monge--Amp\`{e}re equation that involves a natural measure on the normal bundle, and the solution in \cite{CTY24} uses optimal transport techniques from convex analysis. Once pulled back to $X$ and trivialized by the defining section of $D$, these ansatzs can be identified with functions on $X$. 

When $D$ is smooth (and therefore irreducible), \cite{TY90} constructed a complete Calabi--Yau metric on $X$ which is exponentially asymptotic to the Calabi ansatz, and the decay rate was later improved by \cite{Hei10}. When $D$ has two components, \cite{CL24} solved the generalized Calabi ansatz and completed it to a Calabi--Yau metric on $X$. These metrics are of the form $dd^c \psi$ where $\psi$ is a smooth strictly plurisubharmonic function on $X$ that solves the complex Monge--Amp\`{e}re equation \[\CMA(\psi) = \mu.\] 
We will further explain these results in \S\ref{sec:generalized-calabi-ansatz}, and refer to the aforementioned complex-analytic objects as \emph{Archimedean}. 

A smooth complex (quasi-)projective variety $X$, in fact, has both an Archimedean and a non-Archimedean analytification. The former is a complex K\"{a}hler manifold, also denoted by $X$, while the latter is a Berkovich analytic space $X^\na$. If $X$ is Calabi--Yau, the natural analog of the aforementioned canonical measure $\mu$ on $X$ is the Lebesgue measure on what is called the essential skeleton $\sk(X)$ in $X^\na$. We denote this measure as $\mu^\na = \mathrm{Leb}_{\mathrm{Sk}(X)}$ and will study it in more detail in \S \ref{sec:berkovich-space}. By the work of \cite{CLD25}, which we summarize in \S \ref{sec:pp-on-berk-spaces}, one can also make sense of a non-Archimedean Monge--Amp\`{e}re operator $\NAMA(-)$ on $X^\na$. In light of the established results regarding $\CMA(-)$ on $X$, it is natural to consider the analogous question on $X^\na$: does there exist a function $\psi^\na$ on $X^\na$, regular and plurisubharmonic in a suitable sense, such that \begin{equation}\label{eq:NAMA}\NAMA(\psi^\na) = \mu^\na?\end{equation}

The main goal of our work is to show that in certain situations, we can build on the generalized Calabi ansatz to solve (\ref{eq:NAMA}) on the Berkovich space $X^\na$.\footnote{When $X$ is projective and more generally defined over $\C(\!(t)\!)$, this equation was solved in \cite{BFJ15} by a variational method.} Because of the analogy between the canonical measures $\mu$ and $\mu^\na$, we call solutions to (\ref{eq:NAMA}) the \emph{non-Archimedean Calabi--Yau potentials}. 

\begin{alphthm}\label{intro:thm-1}
    Let $n, d$ be integers with $n > d \geq 1$, and let $\bar{X}$ be a smooth projective Fano variety of dimension $n$, with a reduced simple normal crossing anticanonical divisor $D$ whose dual complex is the standard $(d-1)$-simplex. Set $X = \bar{X} \backslash D$, which is an affine Calabi--Yau variety. Then there is a non-Archimedean Calabi--Yau potential, namely a continuous plurisubharmonic function $\psi^\na$ on the Berkovich analytification $X^\na$, that solves \[\NAMA(\psi^\na) = \mu^\na.\] The solution $\psi^\na$ is built on the Calabi ansatz from \cite{Cal79}, \cite{TY90}, \cite{CL24}, and \cite{CTY24}.
\end{alphthm}

Since $\mathcal{O}(D)|_X$ is a trivial line bundle on $X$, the solution $\psi^\na$ in Theorem \ref{intro:thm-1} is a function. We mention that, for $d > 2$, the existence of a global Archimedean Calabi--Yau potential on general $X$ is yet unknown, while Theorem~\ref{intro:thm-1} demonstrates the existence of a global non-Archimedean Calabi--Yau potential on $X^\na$. 

Here, we would also like to point out that both an Archimedean Calabi ansatz and a non-Archimedean (NA) Calabi--Yau potential come from a convex function on $\R^d_+$ obtained in \cite{CTY24}. This close tie between convex geometry and both Archimedean and NA geometries provides hope that NA geometry can inform existence or non-existence of Archimedean Calabi--Yau potentials on general smooth affine varieties. 

Indeed, we will strengthen the connection between Archimedean and NA Calabi--Yau potentials by showing a continuity statement on a suitable space encapsulating both Archimedean and NA geometries. This space, named the hybrid space, has been constructed and studied in \cite{Ber09}, \cite{BJ17}, \cite{Fav20}, \cite{Shi22}, and \cite{LP24}. In simple terms, the hybrid space associated to a quasi-projective variety $X$ is a topological space admitting a map to the unit interval, 
\[\lambda: X^\hyb \to [0,1]\]
such that $\lambda^{-1}(0)$ is the Berkovich analytification of $X$, whereas $\lambda^{-1}(\tau)$ is homeomorphic to the Archimedean (holomorphic) analytification of $X$ for all $\tau \in (0,1]$. By the work of \cite{Fav20} and \cite{PS23}, there is a pluripotential theory on these hybrid spaces, rendering them an appropriate setting for comparing the Archimedean and NA objects of concern. 

To state the hybrid results uniformly, let the triple $(X, \psi^{\mathrm{a}}, \psi^\na)$ denote any of the following: 
\begin{enumerate}[leftmargin=*]
    \item\label{intro:TY} $\bar{X}$ is a smooth projective Fano variety, $D$ is an irreducible smooth anticanonical divisor, $X = \bar{X} \backslash D$, with $\psi^{\mathrm{a}}$ the Tian--Yau potential from \cite{TY90}, and $\psi^\na$ from Theorem \ref{intro:thm-1};
    \item\label{intro:CL} $\bar{X}$ is a smooth projective Fano variety, $D$ is an anticanonical divisor with irreducible components $D_1, D_2$ such that $D_1 \cap D_2$ is smooth irreducible, $X = \bar{X} \backslash D$, with $\psi^{\mathrm{a}}$ the Collins--Li potential from \cite{CL24}, and $\psi^\na$ from Theorem \ref{intro:thm-1};
    \item\label{intro:CA} we can also consider the Calabi model space. Let $D$ be a reduced simple normal crossing anticanonical divisor in some smooth projective Fano variety $Y$, with the requirement that the deepest intersection stratum $Z$ of $D$ is irreducible. Let $\bar{X}$ be the total space of $N_{Z / Y}$, and $X$ is the complement of the zero section of $\bar{X}$. Now $\psi^{\mathrm{a}}$ can be any smooth extension of the generalized Calabi ansatz from \cite{CTY24} that is bounded away from the zero section, and $\psi^\na$ comes from Theorem \ref{intro:thm-1}.  
\end{enumerate}

\begin{alphthm}\label{intro:thm-2}
    In situations (\ref{intro:TY}), (\ref{intro:CL}) and (\ref{intro:CA}), under some explicit scaling $c(\tau)$, the function \[\psi^\hyb = \begin{cases}
        c(\tau) \psi^{\mathrm{a}} & \tau \neq 0\\
        \psi^\na & \tau = 0
    \end{cases}\] on $X^\hyb$ is continuous. 
\end{alphthm}

As a direct consequence of the volume convergence result in \cite{Shi22}, we also have weak convergence of the Monge--Amp\`{e}re measures of these metrics on $X^\hyb$:\[d(\tau) \CMA(\psi^{\mathrm{a}}) = d(\tau) \mu \to \mu^\na = \NAMA(\psi^\na),\] where $d(\tau)$ is again some explicit scaling.  

We now outline the proof strategy of Theorem~\ref{intro:thm-1}. We use tropicalization maps to work with a local pluripotential theory on Berkovich spaces developed systematically in \cite{CLD25}. This allows us to translate the non-Archimedean Monge--Ampère equation into a real Monge--Ampère equation, and also translate the function $\psi^\na$ on the Berkovich space into a convex function $v$ defined on an Euclidean space. Then, we interpret the PDEs studied in \cite{CL24} and \cite{CTY24} as computing the volume of subgradient polytopes associated to the convex function $v$. In addition, we use the boundary condition for the generalized Calabi ansatz, first prescribed in \cite{CL24} and later adapted to the general situation in \cite{CTY24}, to exclude singular charges of the real Monge--Amp\`{e}re measure of $v$. 

As a toy example, for $t > 0$, the volume of the convex hull of the points $\{t^{1/n}e_i\}_{i=1}^n$ is linear in $t$, where $\{e_i\}_{i=1}^n$ is the standard basis of $\mathbf{R}^n$. This convex hull will turn out to be the subgradient polytope of $\psi^\na$ at $\mathrm{Sk}(X)$ when $D$ is irreducible. Interestingly, by the power rule of integration, the anti-derivative of $t^{1/n}$ is a constant multiple of $t^{(n+1)/n}$, whose exponent is exactly the growth rate showing up in the classical Calabi ansatz. This phenomenon generalizes when $D$ has more than one component. Theorem \ref{intro:thm-1} provides justification for the proposed boundary data in \cite{CTY24}, and answers a question of \cite{CL24}.\footnote{This question is not included in the published version \cite{CL24}, but appears as \S 2.8.3 ``Non-Archimedean Meaning?'' in the first arXiv submission. }

We make two remarks on Theorem \ref{intro:thm-2}. First, the appropriately rescaled pointed Gromov-Hausdorff limit of $X$ coincides with the essential skeleton $\mathrm{Sk}(X)$, and the metric on the Gromov-Hausdorff limit is essentially the generalized Calabi ansatz; see \cite[Proposition 5.2]{CL24}. So Theorem \ref{intro:thm-2} should be viewed as a non-Archimedean realization of the pointed Gromov-Hausdorff convergence on the potential level. 

Second, we discuss the connection of the current work to the SYZ and Kontsevich--Soibelman conjectures, which predict that, for certain degeneration $X \to \mathbf{D}^*$ of projective Calabi--Yau varieties, each fiber admits a special Lagrangian torus fibration over the essential skeleton $\mathrm{Sk}(X)$. Yang Li \cite{Li23} has reduced a metric version of the SYZ conjecture to showing that the \cite{BFJ15} solution to the NA MA equation (\ref{eq:NAMA}) is invariant under a natural retraction map on $X^\na$. Building on this, \cite{HJMM24}, \cite{Li24a}, and \cite{AH23} have shown the metric SYZ conjecture holds for a large class of Calabi--Yau hypersurfaces, where a central difficulty was that the affine structure on $\mathrm{Sk}(X)$ has singularities. In our setup, the generalized Calabi ansatz also admits a torus fibration over $\mathrm{Sk}(X)$, but the difficulty in the projective setting is not present since the essential skeleton has a global affine structure. 

In the concluding section \S \ref{sec:concluding}, we remark on how the measure convergence result from \cite{Shi22} can in fact recover the homogeneous degree of the generalized Calabi ansatz obtained in \cite{CL24}. We also discuss Odaka's conjecture \cite{Oda20} which compares volume growth dimension and essential skeleton dimension.
    
\subsection*{Organization} In \S \ref{sec:generalized-calabi-ansatz} we review the construction of the generalized Calabi ansatz. We then collect facts about Berkovich spaces in \S \ref{sec:berkovich-space} and the non-Archimedean Monge--Ampère operator in \S \ref{sec:pp-on-berk-spaces}. In \S \ref{sec:non-archimedean-metric}, we prove Theorem~\ref{intro:thm-1} in five steps. In \S \ref{sec:hybrid-spaces} we review the theory of hybrid spaces. Then we prove Theorem~\ref{intro:thm-2} on appropriate hybrid spaces in \S \ref{sec:hybrid-continuity}. Some remarks are given in \S \ref{sec:concluding}. 

\subsection*{Notation} 
\begin{itemize}
    \item The symbol $[k]$ refers to the ordered set $[1, \cdots, k].$
    \item To denote a point in a Berkovich space, we will sometimes switch between multiplicative notation using multiplicative (semi)norms $|\cdot|$, or additive notation using (semi)valuations $\nu$. They are related by $\nu(\cdot) = -\log |\cdot|.$
    \item For a convex set $\sigma \subseteq \mathbf{R}^m$, we let $\mathbf{L}_\sigma$ denote the smallest affine space in $\mathbf{R}^m$ containing $\sigma$, and let $\mathbf{L}_{\sigma}^\vee$ denote its dual space. 
    \item We use additive notation for line bundles. For example, $L_1 + L_2$ means $L_1 \otimes L_2$.
    \item We often use the term Archimedean to refer to the complex-analytic situation, in order to highlight the contrast with the non-Archimedean one. 
    \item We use the acronyms NA for non-Archimedean, MA for Monge--Amp\`{e}re, FS for Fubini--Study, and SNC for simple normal crossing.
\end{itemize}

\subsection*{Acknowledgment}
I am grateful to my advisor Mattias Jonsson for his invaluable inspiration, encouragement, and feedback throughout the preparation of this work. I thank my academic sister Yueqiao Wu for many fun and exciting discussions. I also thank Tristan Collins, Henri Guenancia, Yang Li, Yuji Odaka, and Valentino Tosatti for explaining their work, sharing their insights, and offering helpful comments. This project was supported by NSF grants DMS-2154380 and DMS-2452797. 

\section{The generalized Calabi ansatz}\label{sec:generalized-calabi-ansatz}

\subsection{General Setup}\label{setup}\hfill

We review the construction and solution of the generalized Calabi ansatz from \cite{CL24} and \cite{CTY24}, which recover the classical Calabi ansatz from \cite{Cal79} and \cite{TY90}. The construction of Calabi–Yau potentials asymptotic to this ansatz, following \cite{TY90} and \cite{CL24}, is deferred to later sections \S \ref{Conv:TY-metric} and \S \ref{Conv:CL-metric} in the course of proving Theorem~\ref{intro:thm-2}. 

Let $\bar{X}$ be a smooth complex projective Fano variety of dimension $n$. Let $L$ be an ample $\mathbf{Q}$-line bundle satisfying $bL = - K_{\bar{X}}$ for some $b \in \mathbf{Q}_{> 1}$. Fix any integer $d \in [1,n-1]$, and any $b_1, \cdots, b_d \in \mathbf{Q}_{> 0}$ such that each $b_iL$ is a $\mathbf{Z}$-line bundle and $b_1 + \cdots + b_d = b$. Fix any reduced simple normal crossing divisor $D = D_1 + \cdots + D_d \in |{-}K_{\bar{X}}|$ with $D_i \in |b_iL|$. 

Because $d \leq n-1$, repeated applications of the Lefschetz hyperplane theorem and adjunction show that the deepest intersection stratum $Z \coloneqq D_1 \cap \cdots \cap D_d$ is a non-empty, smooth, compact, and connected Calabi--Yau variety, while any other intersection stratum is Fano.

\begin{example}
    Let $\bar{X} = \mathbf{P}^{n}, L = \mathcal{O}_{\mathbf{P}^n}(1)$ and $d < n$. Pick integers $b_1, \cdots, b_d \geq 1$ such that $b_1 + \cdots + b_d = n+1$. If we take general members $D_i \in |b_i L|$, then $D = D_1 + \cdots + D_d$ serves as an example of this setup.
\end{example}

\subsection{Calabi model space}\label{sec:calabi-model-space}\hfill

In fact, two Calabi--Yau varieties arise from the pair $(\bar{X}, D)$. As we have just mentioned, the deepest intersection stratum $Z$ of $D$ is a compact Calabi--Yau. On the other hand, the complement $X \coloneqq \bar{X} \backslash D$ is an affine Calabi--Yau variety, because it has a non-vanishing regular $n$-form with a simple pole along $D$, and this form is unique up to scaling.  

While the existence of Calabi--Yau metrics on $Z$ has been well understood since \cite{Yau78}, far less is known on $X$. To construct such metrics on $X$, the first step is to prescribe an ansatz in the region on $X$ near $Z$. In light of the tubular neighborhood theorem, a natural model for this region is the normal bundle $N_{Z / \bar{X}}$, which by the adjunction formula splits as $N_{Z / \bar{X}} = b_1L|_Z \oplus \cdots \oplus b_d L|_Z.$
Let $N_{Z/\bar{X}}^\times$ be the complement of the union of the zero section in each factor $b_iL|_Z$. We denote by $\Omega$ a regular non-vanishing $n$-form on $N_{Z/\bar{X}}^\times$ with a simple pole along the zero section of each $b_iL|_Z$; again $\Omega$ is unique up to scaling. 

Because the first Chern class of $L|_Z$ is ample, by \cite{Yau78} there is some hermitian metric on $L|_Z$, written additively as $-\log h$, whose curvature form is Ricci flat. After taking multiples, we also obtain metrics $-\log h_i$ on each $b_i L|_Z$. 

\begin{definition}
    The \emph{generalized Calabi ansatz} is a smooth convex function \[u\colon \mathbf{R}^d_{> 0} \to \mathbf{R}\] such that the function $\psi^{\mathrm{a}} \coloneqq u(-\log h_1, \cdots, -\log h_d)$, defined on some open subset $U_Z$ of $N_{Z/\bar{X}}^\times$ near $N_{Z / \bar{X}} \backslash N_{Z / \bar{X}}^{\times}$, solves the Monge--Ampère equation, \begin{equation}\label{eq:normal-bundle-MA}(dd^c \psi^{\mathrm{a}})^n = \mathrm{const} \cdot \Omega \wedge \bar{\Omega}.\end{equation}
    The resulting structure $(U_Z, \Omega, dd^c\psi^{\mathrm{a}})$ is called the \emph{Calabi model space}. Here, the superscript $(-)^\mathrm{a}$ indicates an Archimedean, i.e. complex-analytic, object. 
\end{definition}

Because the curvature form of $-\log h$ is Ricci flat, (\ref{eq:normal-bundle-MA}) reduces to 
\begin{equation}\label{eq:normal-bundle-MA-2}(\det D^2u) (\sum_{i=1}^d b_i \frac{\partial u}{\partial t_i})^{n-d} = \mathrm{const}.\end{equation}

When $d = 1$, Equation (\ref{eq:normal-bundle-MA-2}) becomes
\[u'' (u')^{n-1} = \mathrm{const},\]
whose general solutions are of the form \[u = a_1(nt + a_2)^{(n+1)/n} + a_3\] for some constants $a_1 > 0, a_2, a_3$. After specifying these constants appropriately, this is the classical Calabi ansatz from \cite{Cal79} and \cite{TY90}. 

When $d \geq 2$, Equation (\ref{eq:normal-bundle-MA-2}) is a fully nonlinear second order PDE. In \cite{CL24}, Collins--Li solved the case $d = 2$ by reducing it to an ODE via a homogeneity ansatz, and then imposing a suitable boundary condition. Generalizing this boundary condition, Collins--Tong--Yau \cite{CTY24} solved (\ref{eq:normal-bundle-MA-2}) for any $1 < d < n$ using optimal transport techniques, and formulated a Liouville-type conjecture addressing uniqueness.

\begin{theorem}{(\cite{CTY24}, Corollary 1.1, Theorem 1.2, Proposition 5.1)}\label{thm:CTY}\\
    Let $t_i$ be coordinates of $\mathbf{R}^d$ and $c$ any positive constant. The boundary value problem 
    \[ \begin{cases} 
        \det(D^2 u)(\displaystyle\sum_{i=1}^d \displaystyle\frac{\partial u}{\partial t_i})^{n-d} \ = c \text{ on } \mathbf{R}^d_{> 0}\\
        
        \displaystyle\sum_{i=1}^d \displaystyle\frac{\partial u}{\partial t_i} = 0 \text{ on } \partial \mathbf{R}^d_{\geq 0}. \end{cases}\]
admits a positive convex solution in $C^{1,\alpha}(\mathbf{R}^d_{\geq 0}) \cap C^{\infty}(\mathbf{R}^d_{> 0})$, which is additionally homogeneous of degree $\frac{n+d}{n}$. \end{theorem}


\section{Interlude I: Berkovich spaces}\label{sec:berkovich-space}
Berkovich spaces are analogs of complex-analytic spaces over any complete valued field $(k, |\cdot|)$. In particular, they carry a structure sheaf and have a well-developed pluripotential theory. We will only consider Berkovich spaces arising as analytifications of complex algebraic varieties $X$, which we use to capture the degeneration of $X$. For this, we analytify $X$ over the trivially valued field $(\C, |\cdot|_0)$, where $|\cdot|_0$ is the trivial valuation defined by $|0|_0 = 0$ and $|a|_0 = 1$ for any $a \neq 0$. This valued field is non-Archimedean and complete. 

If $X = \spec(A)$, an affine scheme of finite type over $\C$, its Berkovich analytification $\spec(A)^\na$ over $(\C, |\cdot|_0)$ consists of all semivaluations on $A$ extending $|\cdot|_0$, and is equipped with the weakest topology such that evaluation on any $f \in A$ is continuous. For a general complex variety, the analytification is obtained by gluing these affine pieces. Moreover, if $Y$ is projective with a closed subscheme $D$, then there is a natural inclusion $D^\na \subseteq Y^\na$. Letting $U = Y \backslash D$, we have $U^\na = Y^\na \backslash D^\na$. 

\begin{remark}
    In fact, if we let instead $k = \mathbf{C}$ with the usual Euclidean norm (albeit not non-Archimedean), then the same construction gives the usual complex analytification. 
\end{remark}

\subsection{Tropical spectrum over a trivially valued field}\label{sec:trop-spec}\hfill

Let $Y$ be a smooth complex projective variety, and as before fix the base field of analytification to be $(\mathbf{C}, |\cdot|_0)$. In this setting, the Berkovich analytification $Y^\na$ admits an equivalent but global description, developed in \cite{JM12}, \cite{BJ22} and \cite{MP24}.

Let $\mathfrak{I}_Y$ be the collection of coherent ideal sheaves on $Y$, which is a semiring. The tropical spectrum $\mathrm{TSpec}(\mathfrak{I}_Y)$ is the set of all functions \[\chi\colon \mathfrak{I}_Y \to [0, \infty]\]
satisfying \[\chi(\mathcal{O}_Y) = 0 , \chi(0) = \infty,\] and moreover for any two ideals $\mathfrak{a}, \mathfrak{b} \in \mathfrak{I}_Y$, we have
\begin{align}\begin{split}\label{trop-spec}
    \chi(\mathfrak{a} \cdot \mathfrak{b}) = \chi(\mathfrak{a}) + \chi(\mathfrak{b}) &, \space \chi(\mathfrak{a} + \mathfrak{b}) = \min\{\chi(\mathfrak{a}), \chi(\mathfrak{b})\}.
\end{split}\end{align}

There is a partial ordering on $\mathrm{TSpec}(\mathfrak{I}_Y)$, given by \[\chi \leq \nu \iff \chi(\mathfrak{a}) \leq \nu(\mathfrak{a}) \text{ for all } \mathfrak{a} \in \mathfrak{I}_Y.\]

We equip $\mathrm{TSpec}(\mathfrak{I}_Y)$ with the weakest topology such that evaluation on any $\mathfrak{a} \in \mathfrak{I}_Y$ is continuous. This topology is called the \emph{Berkovich topology}. As shown in \cite[\S 1]{MP24}, if we interpret functions $\chi$ as (semi)valuations on $Y$, then there is a natural homeomorphism $i\colon \mathrm{TSpec}(\mathfrak{I}_Y) \xrightarrow[]{} Y^\na$. So hereafter we will use them interchangeably. Moreover, if $D$ is a subscheme with associated ideal sheaf $\mathcal{I}_D$, we often write $\chi(D)$ to denote $\chi(\mathcal{I}_D)$. 

Let $Y^{\mathrm{val}}$ denote the set of all valuations in $Y^\na$. This set $Y^{\mathrm{val}}$ is birationally invariant, because any $\chi \in Y^\val$ induces a valuation on the function field $\C(Y)$, which is birationally invariant. Using the tropical spectrum description, we can characterize 
\[Y^{\mathrm{val}} = \{\chi \in Y^\na: \chi(\mathfrak{a}) = \infty \text{ if and only if } \mathfrak{a} = 0\}.\]

Now let $Z \subseteq Y$ be a closed subvariety and $U \subseteq Y$ an open subvariety. Then the subspaces $Z^\na, U^\na \subseteq Y^\na$ can be characterized by 
\begin{align*}
    Z^\na & = \{\chi \in Y^\na: \chi(\mathfrak{a}) = \infty \text{ for some } \mathfrak{a} \text{ cosupported on } Z\},\\
    U^\na &= \{\chi \in Y^\na: \chi(\mathfrak{a}) < \infty \text{ for all } \mathfrak{a} \text{ cosupported on } X \backslash U\}.
\end{align*}
It is thus clear that, if $U = Y \backslash Z$, then $U^\na = Y^\na \backslash Z^\na$. 

\begin{remark}[Center]\label{rem:center}
    Because $Y$ is projective, any function $\chi \in Y^\na$ admits a unique \emph{center} $c_{Y}(\chi) \in Y$, defined as the sum of all ideals on which $\chi$ is positive, which is prime by \cite[Lemma 2.1.1]{MP24}. When $\chi \in Y^\na$ is a valuation, this coincides with the usual center in the valuative criterion of properness. The center map is functorial, that is, if $f: X \to Y$ is a map, then $c_Y \circ f^\na = f \circ c_X$ where $f^\na: X^\na \to Y^\na$ is the induced morphism. 
\end{remark}

\begin{remark}[Pullback]\label{rmk:pullback-of-val}
Let $f\colon Y \to Z$ be a map of complex varieties. The induced map $f^\na\colon Y^\na \to Z^\na$ can be described as follows. Recall that for any $\mathfrak{a} \in \mathfrak{I}_Z$, there is an inverse image ideal sheaf $\mathfrak{a}' \coloneqq f^{-1} \mathfrak{a} \cdot \mathcal{O}_Y \in \mathfrak{I}_Y$. So for any $\chi \in Y^\na$, we define
\begin{equation*}f^\na(\chi)(\mathfrak{a}) \coloneqq \chi(\mathfrak{a}') \text{ for any }\mathfrak{a} \in \mathfrak{I}_Z.\end{equation*}
\end{remark}

\begin{remark}[Topologies and the structure sheaf]
    There is also a Zariski topology on $Y^\na$, defined by setting $U^\na \subseteq Y^\na$ as open for any $U \subseteq Y$ that is Zariski open. The Zariski topology is weaker than the Berkovich topology. Unless stated otherwise, when we say a subset is open in $Y^\na$, we mean open in the Berkovich topology. 
    
    In addition, the Berkovich space $Y^\na$ carries a coherent structure sheaf $\mathcal{O}_{Y^\na}$ in the Berkovich topology. In particular, we can evaluate $\mathcal{O}_{Y^\na}$ on $U^\na$ for any Zariski open $U \subseteq Y$. 
\end{remark}

\subsection{Quasimonomial valuations}\label{sec:qm}\hfill

We still analytify over $(\C, |\cdot|_0)$ and let $Y$ be a smooth complex projective variety with a simple normal crossing divisor $D$. We define quasimonomial valuations, a natural class of birationally invariant points in $Y^\na$, in a nice way using the tropical spectrum description.

\begin{definition}\label{def:qm}
    Say the irreducible components of $D$ are $D_1, \cdots, D_k$. An intersection stratum of $D$ is an irreducible component of the intersection of some of the $D_i$'s. Let $\eta$ be the generic point of some intersection stratum of $D$. For any $k$-tuple $a = (a_1, \cdots, a_k) \in \mathbf{R}^k_{\geq 0}$, there is a unique minimal element $\chi_{a, \xi} \in \mathrm{TSpec}(\mathfrak{I}_Y)$ that is centered at another generic point $\xi$ of another intersection stratum of $D$, with $\eta \in \bar{\xi}$, and satisfies \begin{equation}\label{eq:qm-val}\chi_{a,\xi}(D_i) = a_i \text{ for all } 1 \leq i \leq k.\end{equation}

    This is a valuation, as opposed to a semivaluation. We denote by $\QM_\eta (Y,D)$ the set of all such valuations. If $\eta$ is the generic point of an intersection stratum $Z \subseteq D$, then we also use the notation $\QM_Z (Y, D)$ to denote $\QM_\eta(Y,D)$. 

    The quasimonomial valuations on $D$ are \[\QM(Y, D) \coloneqq \bigcup_{Z} \QM_Z(Y,D),\] where the union runs through all intersection strata $Z$ of $D$. By the identification of $Y^\na$ with $\mathrm{TSpec}(\mathfrak{I}_Y)$, we will consider $\QM_Z(Y,D)$ and $\QM(Y,D)$ as subsets of $Y^\na$. 
\end{definition}

\begin{remark}\label{rmk:equiv-def-qm}
    This definition of quasimonomial valuations is equivalent to the one in \cite[\S 1.2]{JM12}, which utilizes Cohen’s structure theorem. 
\end{remark}

Let $U \coloneqq Y \backslash D$. Because $U$ is birational to $Y$, we have $U^{\mathrm{val}} = Y^{\mathrm{val}}$. It follows that 

\begin{lemma}\label{lem:qm-in-Uan}
    The elements in $\QM(Y,D)$ is contained in $U^{\mathrm{val}} \subseteq U^\na = Y^\na \backslash D^\na$. 
\end{lemma}

Because $D$ is an SNC divisor, one can associate to it a dual complex $\Delta(D)$ whose faces are, in an order reversing manner, in one-to-one correspondence with the intersection strata of $D$. As explained in \cite[\S 4.2]{JM12}, the set $\QM(Y, D)$ can be naturally realized as a cone over $\Delta(D)$. We now give two examples. 

\begin{example}(\emph{cf.} \cite[2.2.4.]{CLD25})\label{eg:qm-for-torus}
    Let $Y = \mathbf{P}^n$ with homogeneous coordinates $z_0, \cdots, z_n$, and $D_i = \{z_i = 0\}, D = D_0 + \cdots + D_n$. 
    The deepest strata in $D$ are the points $p_i$ with $1$ at the $i$-th coordinate and $0$ at elsewhere. The set $\QM_{p_i}(\mathbf{P}^n, D)$ of quasimonomial valuations is isomorphic to $\R^n_{\geq 0}$ via evaluation at the ideal sheaves associated to $D_0, \cdots, D_{i-1}, D_{i+1}, \cdots, D_n$. Alternatively, observe that the dual complex $\Delta(D)$ is the boundary of the standard $n$-simplex, and $\QM(\mathbf{P}^n, D)$ its cone. The maximal dimensional faces of $\QM(\mathbf{P}^n, D)$ are precisely $\QM_{p_i}(\mathbf{P}^n, D)$. 
\end{example}

\begin{example}
    Let us specialize to Setup \S \ref{setup} by taking $Y = \bar{X}$ and $D = D_1 + \cdots + D_d$ the anticanonical divisor defined therein. The dual complex of $D$ is the standard $(d-1)$-simplex. In particular, it has exactly one maximal face, corresponding to $Z = D_1 \cap \cdots \cap D_d$. One can check \[\QM(\bar{X}, D) = \QM_Z (\bar{X}, D) \simeq \mathbf{R}_{\geq 0}^d.\] The homeomorphism is given by evaluation at the ideal sheaves associated to $D_1, \cdots, D_d$.
\end{example}

\subsection{Essential skeletons}\label{sec:essential-ske}\hfill

Introduced in \cite{KS06}, the essential skeleton is a piecewise affine subset of the Berkovich analytification of a Calabi–Yau variety over $\C(\!(t)\!)$. It can be characterized as the minimizing locus of either the weight function defined in \cite{MN15} or Temkin's metric on the canonical bundle \cite{Tem16}. These constructions apply to a Calabi–Yau variety $X$ over $\mathbf{C}$ as well. In this setting, the weight function and Temkin's metric coincide by \cite{MMS24}, and the essential skeleton is defined as their common minimal locus. When $X$ is projective, the essential skeleton is a singleton corresponding to the trivial valuation, whereas it is more interesting when $X$ is affine. In this paper, we consider such $X$ arising from smooth log Calabi--Yau pairs. 

\begin{convention}
    We say $(\bar{X}, D)$ is a smooth log Calabi--Yau pair if $\bar{X}$ is a smooth projective Fano variety, and $D$ is a reduced SNC divisor such that $K_{\bar{X}} + D = 0$. 
\end{convention}

Let us fix a smooth log Calabi--Yau pair $(\bar{X}, D)$ from now on, and set $X \coloneqq \bar{X} \backslash D$. Observe that $X$ has a regular non-vanishing $n$-form $\Omega$ with a simple pole along $D$, which is unique up to scaling. In other words, $X$ is a quasi-projective Calabi--Yau variety. We now describe the definition of the essential skeleton $\sk(X) \coloneqq \sk(\bar{X}, D) \subseteq X^\na$. 

The Berkovich analytification $X^\na$ is isomorphic to $\bar{X}^\na \backslash D^\na$, and is independent of the compactification $\bar{X}$. By \cite{MMS24}, the weight function on $X^\na$ is given by $A_{(\bar{X},D)}(\chi) \coloneqq A_{\bar{X}}(\chi) - \chi(D)$, where $A_{\bar{X}}$ is the log discrepancy function extended to the whole Berkovich space $X^\na$ as in \cite[Appendix A]{BJ23}. The minimizing locus of $A_{(\bar{X},D)}$ is precisely the set of quasimonomial valuations $\QM(\bar{X}, D)$; see \cite[Proposition 3.2.5]{Blu18} for a proof. Therefore, \[\mathrm{Sk}(X) = \QM(\bar{X}, D),\] which lies in $X^\na$ by Lemma~\ref{lem:qm-in-Uan}.  

\begin{example}\label{eg:ske-torus}
    We consider a more concrete example which will be used later. Let $\bar{X} = \mathbf{P}^n$ with homogeneous coordinates $z_0, \cdots, z_n$, and $D = \{z_0 \cdots z_n = 0\}$. Then $\bar{X} \backslash D$ is isomorphic to the split torus $\mathbf{G}^n_m$. It comes with a non-vanishing volume form $\tfrac{dz_1}{z_1} \wedge \cdots \wedge \tfrac{dz_n}{z_n}$ and is thus Calabi--Yau. Its essential skeleton is $\mathrm{Sk}(\mathbf{G}^n_m) = \QM(\bar{X}, D) \subseteq (\mathbf{G}^n_m)^\na$, already studied in Example~\ref{eg:qm-for-torus} and shown there to be isomorphic to $\mathbf{R}^n_{\geq 0}$.
\end{example}

\section{Interlude II: The Monge--Ampère operator on Berkovich spaces}\label{sec:pp-on-berk-spaces}
For later use, we review the construction of the Monge--Ampère operator on Berkovich spaces, as developed in \cite{CLD25}. Throughout this section, $Y$ denotes a smooth connected complex variety of dimension $n$, not necessarily projective, and $Y^\na$ its Berkovich analytification over $(\C, |\cdot|_0)$. 

\subsection{Moment maps and tropical charts}\label{sec:moment-tropical}\hfill

One of the central ideas of \cite{CLD25} is to regard certain maps to Euclidean spaces as tropical charts of $Y^\na$. Indeed, Lagerberg \cite{Lag12} equipped Euclidean spaces with a pluripotential theory that shares many features with the classical complex one, thereby allowing a parallel theory to be developed on $Y^\na$ through tropical charts. 

\begin{convention}
    We use $\mathbf{T} = \mathbf{T}^k$ to denote a torus isomorphic to $\mathbf{G}_m^k = \spec(\mathbf{C}[T_1^\pm, \cdots, T_k^\pm])$ for some $k > 1$, but without fixing an isomorphism. 
        
    In practice, we will use $\mathbf{T}$ of the form $\mathbf{G}^{k+1}_m / \mathbf{G}_m$, where $\mathbf{G}_m$ acts coordinate-wise. This is natural because we will embed $Y$ into some projective space $\mathbf{P}^k$ by sections of a line bundle, and $\mathbf{G}^{k+1}_m / \mathbf{G}_m$ is the complement of coordinate hyperplanes in $\mathbf{P}^k$. 
\end{convention}

\begin{definition}\label{def:tropicalization}
    Let $U$ be open in $Y^\na$. A morphism of Berkovich spaces $\tau\colon U \to \mathbf{T}^\na$ is called a moment map. 
\end{definition}

Let $M$ denote the group of characters $\mathbf{T} \to \mathbf{G}_m$, and let $N \coloneqq \mathrm{Hom}(M, \mathbf{Z})$. Then the vector space $N_{\mathbf{R}}(\mathbf{T}) \coloneqq N \otimes_{\mathbf{Z}} \mathbf{R} = \mathrm{Hom}(M, \mathbf{R})$ has a natural integral affine structure. Any semivaluation $\nu \in \mathbf{T}^\na$ defines an element in $N_{\mathbf{R}}(\mathbf{T})$ by sending $f \in M$ to $\nu(f)$, which is finite because $f \in M$ is invertible. This assignment does not depend on the coordinates on $\mathbf{T}$; it is called the \emph{tropicalization map} and is denoted as $\trop \colon \mathbf{T} \to N_{\mathbf{R}}(\mathbf{T}).$ 

\begin{example}
    If $\mathbf{T} = \mathbf{G}_m^k$ with affine coordinates $T_1, \cdots, T_k$ on $\mathbf{T}$, then $N_{\R}(\mathbf{T})$ is canonically isomorphic to $\mathbf{R}^k$ by identifying $\alpha \in N_{\mathbf{R}}(\mathbf{T})$ with $(\alpha(T_1), \cdots, \alpha(T_k)) \in \mathbf{R}^k$. 
\end{example}

\begin{example}\label{eg:alt-trop}
If $\mathbf{T} = \mathbf{G}^{k+1}_m / \mathbf{G}_m$ with homogeneous coordinates $z_0, \cdots, z_k$, then $N_{\mathbf{R}}(\mathbf{T})$ is canonically isomorphic to $\mathbf{R}^{k+1}/ \mathbf{R}$ by identifying $\alpha \in N_{\mathbf{R}}(\mathbf{T})$ with $(\alpha(z_0), \cdots, \alpha(z_n))$. Here $\mathbf{R}$ acts additively on $\mathbf{R}^{k+1}$.  
\end{example}

With the tropicalization map, we can define tropical charts. Again, let $U$ be open in $Y^\na$, and let $\tau\colon U \to \mathbf{T}^\na$ be a moment map. 

\begin{definition}
    The composition $\tau^\trop \coloneqq \trop \circ \tau\colon U \to \NRT$ is a tropical chart of $Y^\na$.
\end{definition}

\subsection{Functions and metrics}\hfill

Let $U$ be open in $Y^\na$, and let $\tau\colon U \to \NRT$ be a tropical chart. Additionally, let $L$ be a line bundle on $Y$. This induces a line bundle $L^\na$ on $Y^\na$. We also use $L^\na$ to denote its total space, which coincides with the Berkovich analytification of the total space of $L$.

\begin{definition}\label{def:class-of-functions}
Let $g \colon \NRT\to \mathbf{R}$ be any function. If $g$ is smooth or convex, then $(\tau^\trop)^* g$ is a smooth or plurisubharmonic tropical function on $U$, respectively. If a function $\psi$ on $U$ can be written as $(\tau^\trop)^* g$ for some $g$, we say $\psi$ is tropicalizable. 
\end{definition}

As developed in \cite{CLD25}, there is a sheaf $\mathrm{CPSH}$ of continuous plurisubharmonic functions, which is closed under locally uniform limits. In particular, $\mathrm{CPSH}(U)$ contains all plurisubharmonic tropical functions for $U \subseteq Y^\na$ open.

\begin{example}(\emph{cf.}\cite[Corollarie 8.2.4]{CLD25})\label{eg:comp-of-fn-is-psh}
    Let $U$ be open in $Y^\na$, and let $f_1, \cdots,f_k \in \mathcal{O}_{Y^\na}(U)^{\times}$ be invertible regular functions. Let $g \colon \mathbf{R}^k \to \mathbf{R}$ be a convex function. Then the function $g(\log|f_1|, \cdots, \log|f_k|)$ is in $\mathrm{CPSH}(U).$
\end{example}

\begin{example}(\emph{cf.} \cite[Corollaire 8.2.3]{CLD25})\label{eg:max-of-psh}
    Let $U$ be open in $Y^\na$, and let $f_1, \cdots, f_k$ be any finite collection of functions in $\mathrm{CPSH}(U)$. Then $\max(f_1, \cdots, f_k)$ is in $\mathrm{CPSH}(U)$ as well. 
\end{example}

We now move on to metrics on the line bundle $L^\na$. Let $(L^\na)^\times$ denote the total space of $L^\na$ with the zero section removed. We also use the notation $\overline{\mathbf{R}} \coloneqq \mathbf{R} \cup \{-\infty\}$, where $-\infty$ has a basis of open neighborhoods given by the intervals $[-\infty, a)$ for $a \in \mathbf{R}$. 

\begin{definition}
    A continuous metric on $L^\na$ is a continuous function $\phi\colon (L^\na)^\times \to \overline{\mathbf{R}}$ whose restriction to each fiber is a norm.\footnote{We refer to \cite[\S 17.1]{CLD25} for the exact meaning of this; see also \S 2.2 of the first arXiv version of \cite{BJ22}.} A continuous metric $\phi$ is plurisubharmonic if, for any $U \subseteq Y$ Zariski open and any invertible section $s \in H^0(U, L)$, the function $\phi \circ s\colon U^\na \to \overline{\mathbf{R}}$ is in $\mathrm{CPSH}(U^\na)$. It is singular if it takes the value $-\infty$, and non-singular otherwise. 
\end{definition}

If $\phi_i$ is a metric on $L_i$ for $i = 1, 2$, then $\phi_1 + \phi_2$ is a metric on $L_1 + L_2$. Similarly, if $\phi$ is a metric on $L$, then it induces a metric on $eL$ for any $e \in \mathbf{N}$. In addition, a metric on the trivial line bundle $\mathcal{O}_{Y^\na}$ is simply a function on $Y^\na$. 

\begin{example}
    Let $s' \in H^0(Y, L)$ be a section. Then $\phi \coloneqq \log|s'|$ is a metric given by $\phi \circ s = \log|s' / s|$ for any non-vanishing local section $s$ of $L$. It is the unique metric satisfying $\phi \circ s' = 0$. If $s'$ is non-vanishing, then $\phi$ is non-singular. 
\end{example}

\begin{definition}[Fubini--Study metrics and the trivial metric]\label{def:triv-metric-over-triv-valued-field}\hfill

    Let $k \in \mathbf{Z}$ and pick global sections $s_0', \cdots, s_m' \in H^0(X, kL)$ without common zeros. Let $\phi$ be the metric on $L^\na$ given by \[\phi \coloneqq k^{-1} \max_{0 \leq j \leq m} (\log|s_j'|).\] 
    It is a continuous plurisubharmonic metric on $L^\na$ by Proposition~\ref{prop:fs-psh} below. 
    
    If $Y$ is projective and $U = Y$, the Fubini--Study metric in fact does not depend on the choice of global sections $s_0', \cdots, s_m'$ and is \emph{the trivial metric} $\phi_{\mathrm{triv}}$, which can be characterized as follows. For any $\nu \in Y^\na$, we have $(\phi_{\mathrm{triv}} \circ s)(\nu) = 1$ for any non-vanishing section $s$ over any Zariski open neighborhood of the center of $\nu$ in $Y$. If $U \subseteq Y^\na$ is open, we also use $\phi_{\mathrm{triv}}$ to denote the restriction of the trivial metric on $L^\na$ to $U$. 
    The notation \[\nu(s) \coloneqq (\phi_{\mathrm{triv}} - \log|s|)(\nu) \in [0, \infty]\] is often used in the literature. \end{definition}

\begin{proposition}\label{prop:fs-psh}
    Any Fubini--Study metric $\phi$ is a continuous plurisubharmonic metric on $L^\na$.
\end{proposition}

\begin{proof}
    Because $\mathrm{CPSH}$ is a sheaf, it suffices to find an open cover $\{U_J\}_J$ of $Y^\na$ such that, for any local section $s$ of $L^\na$ defined on an open $U \subseteq Y^\na$, the function \[\phi \circ s = k^{-1} \max_{0 \leq j \leq m} (\log|s_j'/s|)\] is in $\mathrm{CPSH}(U \cap U_J)$ for all $J$. 
    
    For any non-empty subset $J \subseteq [m]$, define an open subset of $Y^\na$, \[U_J \coloneqq \bigcap_{j \in J} \{\nu(s_j') < 2\} \cap \bigcap_{j \not\in J} \{\nu(s_j') > 1\}.\] Because $\min_{0 \leq j \leq m} \nu(s_j') = 0$ for any $\nu \in Y^\na$, the sets $\{U_J\}_J$ form an open cover of $Y^\na$. For any local section $s$ of $L^\na$ defined over an open $U \subseteq Y^\na$, on $U \cap U_J$ we can write \[\phi \circ s = k^{-1} \max_{j \in J} (\log|s_j'/s|),\] which is in $\mathrm{CPSH}(U \cap U_J)$ by Example~\ref{eg:max-of-psh}. By the sheaf property $\phi \circ s$ is in $\mathrm{CPSH}(U)$, as desired. 
\end{proof}
    
\subsection{Characteristic polyhedron}\label{sec:char-poly}\hfill

As before, $Y$ is a smooth complex variety of dimension $n$. In order to define and compute the non-Archimedean Monge--Ampère operator on $Y^\na$ using tropical charts, one needs to address the subtlety that images of these charts are complicated. In view of this, \cite{CLD25} introduced the notion of a characteristic polyhedron, which is a simplicial subset of $Y^\na$. 

\begin{definition}
    Let $U$ be open in $Y^\na$, with a moment map $\tau\colon U \to \mathbf{T}^\na$. Then $\Sigma_\tau \subseteq U$ is the union of $q^{-1}(\sk(\mathbf{G}^n_m))$ over all algebraic morphisms $q \colon \mathbf{T}^\na \to (\mathbf{G}^n_m)^\na$, where $\sk(\mathbf{G}^n_m) \simeq \mathbf{R}^n$ is the essential skeleton of $(\mathbf{G}^n_m)^\na$ from Example~\ref{eg:ske-torus}.  
\end{definition}

\begin{remark}\label{rem:char-poly-by-proj}
    By \cite[Proposition 12.6.3]{CLD25}, if $\dim \mathbf{T} < n$, then $\Sigma_\tau = \emptyset$. Otherwise, it suffices to consider all morphisms $q$ arising from coordinate projections. Particularly, if $\mathbf{T} = \mathbf{G}^{k+1}_m / \mathbf{G}_m$, then it suffices to use projections $\mathbf{G}^{k+1}_m / \mathbf{G}_m \to \mathbf{G}^I_m / \mathbf{G}_m$ for all subsets $I \subseteq [k+1]$ of size $n+1$. 
\end{remark}

As we will see shortly, the characteristic polyhedron is essentially where integration will take place on $Y^\na$. However, it is a polyhedral complex, and following \cite{CLD25} one needs to decompose it prior to integrating; such a decomposition always exist by Lemme 14.1.5 \textit{loc.\ cit.} We will manually construct one when proving Theorem~\ref{intro:thm-1}. 

\begin{definition}\label{def:cell-decomp}(\emph{cf.} \cite[2.2.4]{CLD25})
    A cellular decomposition $\mathcal{C}$ of $\Sigma_\tau$ consists of a finite collection of closed cells covering $\Sigma_\tau$ such that
    \begin{enumerate}[label=(\alph*), ref=\alph*]
        \item $\tau^\trop$ maps any maximal dimensional cell $C$ homeomorphically onto its image $D$, and $\tau^\trop \colon (\tau^\trop)^{-1}(D^\circ) \to D^\circ$ is a finite covering map; \label{dcp-i}
        \item for any $C, D \in \mathcal{C}$, the intersection $C \cap D$ is a finite union of cells in $\mathcal{C}$; \label{dcp-ii}
        \item for any $C \in \mathcal{C}$, its boundary is a finite union of cells in $\mathcal{C}$; \label{dcp-iii}
        \item for any $C, D \in \mathcal{C}$, if $C \subseteq D$, then either $C = D$ or $C \subseteq \partial D$. \label{dcp-iv}
\end{enumerate}

We use $\mathcal{C}_k$ to denote the set of $k$-dimensional cells in $\mathcal{C}$. 
\end{definition}

\subsection{Non-Archimedean Monge--Amp\`{e}re operator}\label{sec:NAMA}\hfill 

We remain in the setting above, and recall that $\mathrm{CPSH}$ is the sheaf of continuous plurisubharmonic functions on $Y^\na$. By applying a Bedford--Taylor argument, \cite[\S 8.8]{CLD25} constructed a non-Archimedean Monge--Amp\`{e}re operator \[\NAMA\colon \mathrm{CPSH} \to \{\text{positive Radon measure on }Y^\na\},\] which is continuous under locally uniform limits.

The image of a function under this operator is its associated \emph{non-Archimedean Monge--Ampère measure}. These measures do not put mass on $Z^\na \subseteq Y^\na$ for any proper Zariski closed subset $Z \subseteq Y$. Therefore, if $U$ is a non-empty Zariski open subset of $Y$ and $\psi \in \mathrm{CPSH}(Y^\na)$, then 
\begin{equation}\label{cor:no-pluripolar-mass}
    \NAMA(\psi)= \NAMA(\psi)|_{U^\na} = \NAMA(\psi|_{U^\na}) \text{ as measures on $Y^\na$.}
\end{equation}

For certain classes of tropical CPSH functions, their associated measures can be described explicitly as follows. Fix a Zariski open subset $U \subseteq Y$, and pick some $\psi \in \mathrm{CPSH}(U^\na)$ which can be written as $\psi = (\tau^\trop)^* v$ where $\tau\colon U^\na \to \mathbf{T}^\na$ is a moment map and $v\colon \NRT \to \mathbf{R}$ is a smooth convex function. Recall that $\Sigma_\tau$ is the characteristic polyhedron, for which we fix a cellular decomposition $\mathcal{C}$. 

\begin{definition}\label{def:measure-on-each-cell} 
    For each $n$-dimensional cell $C \in \mathcal{C}_n$, let $\sigma = \tau^\trop(C)$ and let $\mathbf{L}_\sigma$ be the $n$-dimensional affine space generated by $\sigma$ in $\NRT$. The space $\mathbf{L}_\sigma$ has a canonical integral affine structure coming from \[N \subseteq N\otimes_{\mathbf{Z}}\mathbf{R} = \NRT.\] Therefore, there is a Lebesgue measure on $\mathbf{L}_\sigma$. 
    
    For simplicity set $v_\sigma \coloneqq v|_{\mathbf{L}_\sigma}$. Note that $\tau^\trop$ restricts to a homeomorphism between $C$ and $\sigma$, so $(\tau^\trop|_C)^{-1}$ is a well-defined function on $\sigma$. We define a measure \[\mu_{\sigma,v} \coloneqq (\tau^\trop|_C^{-1})_* (\mathbf{1}_\sigma \RMA(v_{\sigma}))\] supported on $C$, where $\mathbf{1}_{\sigma}$ is the characteristic function of $\sigma$, and $\RMA(-)$ is the real Monge--Ampère operator on $\mathbf{L}_\sigma$.
    
    The non-Archimedean Monge--Ampère measure associated to $\psi$ is \begin{equation}\label{eq:nama-cont}\NAMA(\psi) \coloneqq \sum_{\sigma} \mu_{\sigma,v},\end{equation} where the sum runs over $\sigma = \tau^\trop(C)$ for all $C \in \mathcal{C}_n$. In particular, its support is contained in $\Sigma_\tau$. This is the starting point taken in \cite[\S 8]{CLD25} for defining the operator $\NAMA(-)$. 
\end{definition}

If we assume that $v$ is only convex, as opposed to being smooth convex, then by continuity of $\NAMA(-)$, for any smooth convex locally uniform approximation $v_k$ of $v$, we have \[\NAMA(\psi) = \lim_{k \to \infty} \sum_{\sigma} \mu_{\sigma, v_k}.\] In fact, because $v_k$ and $v$ are convex, it suffices to assume pointwise convergence. In particular, this does not depends on the chosen approximation, and the support of $\NAMA(\psi)$ is again contained in $\Sigma_\tau$. 

\subsection{Real Monge--Ampère operator and comparison}\label{sec:comparsion}\hfill

It is however desirable to obtain a more explicit formula for $\NAMA(\psi)$ that avoids the use of limits. To this end, we first recall the notion of the Alexandrov Monge--Amp\`{e}re operator. 

\begin{definition}
    Let $\Omega$ be any convex open subset of $\R^n$ and $v\colon \Omega \to \mathbf{R}$ a convex function. The subgradient of $v$ at any point $x \in \Omega$ is the set $\partial v(x) \coloneqq \{p \in \mathbf{R}^n\colon p \cdot (y-x) + v(x) \leq v(y) \text{ for all }y \in \Omega\}.$ The Alexandrov Monge--Amp\`{e}re measure of $v$ is a Borel measure given by \[\RMA(v)(E) \coloneqq |\displaystyle\bigcup_{x \in E}\partial v(x)| \text{ \space \space for any $E \subseteq \Omega$ Borel}.\] A change of variables readily shows that this operator does agree with the classical one on smooth functions. 
\end{definition}

\begin{remark}\label{rem:continuity-of-RMA}
    Suppose $v_k , v$ are convex functions on $\Omega$. If $v_k \to v$ on $\Omega$, then $\RMA(v_k)$ converges to $\RMA(v)$ weakly. 
\end{remark}

Due to the presence of the characteristic function $\mathbf{1}_{\sigma}$ in the measure $\mu_{C,v_k}$ from Definition~\ref{def:measure-on-each-cell}, where $\sigma$ is closed in $\mathbf{L}_{\sigma}$, one can not directly invoke Remark \ref{rem:continuity-of-RMA} to say that $\NAMA(\psi)$ in (\ref{eq:nama-cont}) equals \[\sum_{\sigma} \mu_{\sigma, v},\] since some mass can be trapped outside of $\sigma$, as illustrated by the following example. 

\begin{example}
    Define $v\colon \mathbf{R}^2 \to \mathbf{R}$ by $\max\{x, y, 0\}$ and choose $v_k = k^{-1} \log(\exp(kx) + \exp(ky) + 1),$ which locally uniformly converges to $v$ as $k \to \infty$. Let $E$ be the closed set $\{x, y \leq 0\}$. Then the weak limit of $\mathbf{1}_E \RMA(v_k)$ is $\frac{1}{6} \delta_0$, while $\mathbf{1}_E \RMA(v) = \RMA(v) = \frac{1}{2} \delta_0$, where $\delta_0$ is the Dirac mass at $(0,0)$. 
\end{example}

We will get around this issue in Step \ref{step5} of the next section by choosing suitable approximations $v_k$ of $v$. 

\section{Non-Archimedean Calabi--Yau potentials}\label{sec:non-archimedean-metric}

Let $\bar{X}, D, X$ be as in Setup \ref{setup}, recalled briefly here: $\bar{X}$ is a smooth $n$-dimensional complex projective Fano variety, $D = D_1 + \cdots + D_d$ is a reduced SNC anticanonical divisor in $\bar{X}$ with $d < n$, whose deepest intersection $Z = \bigcap_{i=1}^d D_i$ is irreducible. For $1 \leq i \leq d$, let $s_i \in H^0(\bar{X}, b_i L)$ be the defining equation of $D_i$, where $L$ is an ample $\mathbf{Q}$-line bundle satisfying $-K_{\bar{X}} = \sum_{i=1}^d b_i L$. Fix some positive integer $e$ divisible by $b_1, \cdots, b_d, \sum_{i=1}^d b_i$ such that $eL$ is very ample. 

The complement $X = \bar{X} \backslash D$ is a smooth affine Calabi--Yau variety because it has a non-vanishing regular $n$-form $\Omega$ with a simple pole along $D$, and this form is unique up to scaling. This induces a measure \[\mu^{\mathrm{a}} \coloneqq (\sqrt{-1})^{n^2}\Omega \wedge \bar{\Omega} \text{ on } X.\] As will be reviewed later in \S \ref{sec:meas-convergence}, an appropriately rescaled limit of $\mu^{\mathrm{a}}$ is non-Archimedean in nature; more precisely it is the Lebesgue measure on the essential skeleton $\sk(X)$ in the Berkovich analytification $X^\na$. Therefore, the latter measure is a natural non-Archimedean replacement for $\mu^{\mathrm{a}}$, denoted by $\mu^\na$. 

We call any smooth strictly plurisubharmonic solution to the Archimedean Monge--Amp\`{e}re equation \[\CMA(\psi^{\mathrm{a}}) = \mu^{\mathrm{a}}\] a Calabi--Yau potential, as its curvature form $dd^c \phi^{\mathrm{a}}$ is Ricci flat. It is then natural to seek a (continuous) plurisubharmonic function $\psi^\na$ on $X^\na$ that solves the non-Archimedean Monge--Amp\`{e}re equation \begin{equation}\label{eq:NA-MA}\NAMA(\psi^\na) = \mu^\na,\tag{NA MA}\end{equation} and it is exactly in this sense we call $\psi^\na$ a \emph{non-Archimedean Calabi--Yau potential}. 

The solution crucially uses the PDE from Theorem~\ref{thm:CTY}. Indeed, we will prove the function, \begin{equation}\label{eq:Psi} \psi^\na(\nu) \coloneqq u(\frac{e}{b_1} \nu(s_1), \cdots, \frac{e}{b_d}\nu(s_d)) \text{ \ on $X^\na$},\end{equation} solves (\ref{eq:NA-MA}), where $u$ is the solution of the PDE from Theorem~\ref{thm:CTY} with the constant $c$ therein to be specified later in Step~\hyperref[step2]{2}, and the notation $\nu(s_i)$ is as in Definition~\ref{def:triv-metric-over-triv-valued-field}. 

There have been fruitful interactions between the complex, real, and non-Archimedean Monge--Ampère equations in the literature. For example, on compact toric varieties, a similar idea of using Archimedean Monge--Ampère solutions to produce non-Archimedean solutions was employed in \cite{BGGJK21}, and \cite{Liu11} also used the real Archimedean Monge--Ampère equation to solve the non-Archimedean equation for totally degenerate abelian varieties. Recent advances on the metric SYZ conjecture also relies on a NA MA--real MA comparison property; see \cite{Li23}, \cite{HJMM24}, \cite{Li24a}, and \cite{PS25}.

\subsection*{Outline} We first summarize our proof strategy. 
\begin{itemize}
    \item[Step]\hyperref[step1]{1}: Construct a moment map $\tau$ and a convex function $v$ such that the candidate solution $\psi^\na$ \\ \hspace{\labelwidth}\phantom{1:-}equals $(\tau^\trop)^* v$. In particular, $\psi^\na$ is a continuous plurisubharmonic function on $X^\na$.
    \item[Step]\hyperref[step2]{2}: Find an explicit cellular decomposition of the characteristic polyhedron $\Sigma_\tau$ using quasi- \\ \hspace{\labelwidth}\phantom{1:-}monomial valuations.
    \item[Step]\hyperref[step3]{3}: Rewrite integration formulae on $X^\na$ using the cellular decomposition from Step \hyperref[step2]{2}, thereby \\ \hspace{\labelwidth}\phantom{3:} transforming (\ref{eq:NA-MA}) into a real MA equation.
    \item[Step]\hyperref[step4]{4}: Show the real MA measure of (appropriately restricted) $v$ is the Lebesgue measure on \\ \hspace{\labelwidth}\phantom{1:} $\tau^\trop(\mathrm{Sk}(X))$, the tropical image of the essential skeleton. 
    \item[Step]\hyperref[step5]{5}: Prove the equation from Step \hyperref[step3]{3} using Step \hyperref[step4]{4}. A subtlety is that $\tau^\trop(\Sigma_\tau)$ has a boundary,\\ \hspace{\labelwidth}\phantom{5:} and one needs to find nice approximations $v_k$ of $v$ such that the mass of $\RMA(v_k)$ \\ \hspace{\labelwidth}\phantom{5:} concentrates on $\tau^\trop(\Sigma_\tau)$. 
\end{itemize}

\subsection*{Step 1: Constructing the tropicalization}\label{step1}\hfill

We first construct a moment map $\tau\colon U^\na \to \mathbf{T}^{\na}$ for some Zariski open subset $U \subseteq X$ and some torus $\mathbf{T} = \mathbf{G}^{N+1}_m / \mathbf{G}_m$, such that the characteristic polyhedron $\Sigma_\tau$ admits a simple cellular decomposition in Step \hyperref[step2]{2}. Then we show $\tau$ tropicalizes $\psi^\na|_{U^\na}$ and $\psi^\na$ is in $\mathrm{CPSH}(X^\na)$.

\subsubsection*{(1a) The moment map.}\label{(1a)} \hfill

Because $eL$ is very ample, by Bertini's theorem, after possibly replacing $e$ by a multiple, we can choose a finite subset $\mathcal{S}' \coloneqq \{s_0', \cdots, s_m'\} \subseteq H^0(\bar{X}, eL)$ with $m \geq n$, such that
\begin{enumerate}[label=(\roman*), ref=\roman*]\itemindent=-10pt
    \item $\mathcal{S} \coloneqq \{s_0', \cdots, s_m', s_1^{e/b_1}, \cdots, s_d^{e/b_d}\}$ is a basis of $H^0(\bar{X}, eL)$.
    \item the intersection of the zero loci of any $k$ members of $\mathcal{S}$ is smooth and connected if $k \in [1,n-1]$, reduced and of dimension $0$ if $k = n$, and empty if $k \geq n+1$. \label{tau-i}\label{tau-ii}\label{tau-iii}
\end{enumerate}

This can be achieved by induction as follows. Let us start with $\mathcal{S}_0 \coloneqq\{s_1^{e/b_1}, \cdots, s_d^{e/b_d}\} \subseteq V \coloneqq H^0(\bar{X}, eL).$ 

\begin{lemma}The sections of $\mathcal{S}_0$ are linearly independent.
\end{lemma}

\begin{proof}
    By assumption, the deepest intersection $\bigcap_{s \in \mathcal{S}_0} \mathrm{div}(s)$ is nonempty and connected, and we let $\eta$ denote its generic point. Then $\mathcal{O}_{\bar{X}, \eta}$ is a regular local ring of dimension $d$. By Cohen's structure theorem, there is a non-canonical isomorphism $\widehat{\mathcal{O}_{\bar{X}, \eta}} = \kappa(\eta) [\![z_1, \cdots, z_d]\!]$ where each $z_i$ cuts out $D_i$ at $\eta$. Fix some local trivialization $t$ of $eL$ at $\eta$. If, by way of contradiction, we assume there is some non-trivial linear relation $\sum_i \theta_i s_i^{e/b_i} = 0$, then we get a relation $\sum_i \theta_i (s_i^{e/b_i} / t) = 0 \in \mathcal{O}_{\bar{X}, \eta}$. This implies a non-trivial algebraic relation among $z_1, \cdots, z_d$, which is a contradiction to $\mathcal{O}_{\bar{X}, \eta}$ being regular of dimension $d$. 
\end{proof}

The base case of induction is established as follows. By Bertini's theorem there is an open dense subset $U_0$ of $V$ such that the zero locus of any $s \in U_0$ intersects $E_0 \coloneqq \bigcap_{\sigma \in \mathcal{S}_0} \mathrm{div}(\sigma)_{\mathrm{red}}$ transversally. Moreover, $\mathrm{div}(s) \cap E_0$ is reduced and of dimension $0$ if $d = n - 1$, and is irreducible if $d < n - 1$. Because the quotient map $\pi\colon V \to V / \mathrm{Span}( \mathcal{S}_0)$ is open and continuous, we can assume $s \in U_0$ is not in $\mathrm{Span}(\mathcal{S}_0)$, and define a new set $\mathcal{S}_1 \coloneqq \mathcal{S}_0 \cup \{s_0'\}$. Now $\mathcal{S}_1$ satisfies (ii). This finishes the base step of induction. 

We now induct on the cardinality $|\mathcal{S}_i|$: in each step of induction we will get a bigger set of linearly independent sections $\mathcal{S}_{i+1} \supsetneq \mathcal{S}_{i}$ in $V$, and induct till the cardinality reaches $\dim V$. Suppose we have $\mathcal{S}_i$ that satisfies condition (ii). If $|\mathcal{S}_i| \leq n - 1$, then we can repeat the base case. If $|\mathcal{S}_i| \geq n$, then for each size $n-1$ subset $I \subseteq \mathcal{S}_i$, we get an open dense subset $U_I$ of $V$ such that the zero locus of any element in $U_I$ intersects $\bigcap_{i \in I} \mathrm{div}(s_i)_{\mathrm{red}}$ at a smooth and connected subvariety. Pick some $s \in \bigcap_I (U_I \backslash \mathrm{Span}(\mathcal{S}_i))$, which is possible because the quotient map $V \to V / \mathcal{S}_i$ is open and continuous. Moreover, because $eL$ is very ample, if we let $\mathcal{P}$ be the set of points that arise as the intersection of the zero loci of any $n$ members of $\mathcal{S}_i$, then we can assume $s(p) \neq 0$ for any $p \in \mathcal{P}$. Then $\mathcal{S}_{i+1} \coloneqq \mathcal{S}_i \cup \{s\}$ satisfies condition (ii). This finishes the induction. 

In addition, because $m \geq n$, by (\ref{tau-iii}), the sections in $\mathcal{S}' = \mathcal{S} \backslash \mathcal{S}_0$ have no common zeros. 

Set $N \coloneqq m + d$, and define $\tau_0 \colon \bar{X} \to \mathbf{P}^N$ by $[s_0': \cdots:s_m':s_1^{e/b_1}: \cdots, s_d^{e/b_d}]$, whose restriction to the Zariski open subset \[U \coloneqq \big(\bigcap_{j=0}^m\{s_j' \neq 0\} \big) \cap \big(\bigcap_{i = 1}^d \{s_i \ne 0\}) = \big(\bigcap_{j=0}^m\{s_j' \neq 0\} \big) \cap X\] lands in the torus $\mathbf{T} = \mathbf{G}^{N+1}_m / \mathbf{G}_m$ of $\mathbf{P}^{N}$; that is, the complement of all coordinate hyperplanes in $\mathbf{P}^N$. Now the analytification of $\tau_0|_U$, denoted as \[\tau\colon U^\na \to \mathbf{T}^\na,\] is a moment map. Recall from Example \ref{eg:alt-trop} that there is a tropicalization map $\trop \colon \mathbf{T}^\na \to \NRT = \mathbf{R}^{N+1} / \mathbf{R}$, and $\tau^\trop = \trop \circ \tau$ is a tropical chart.

\subsubsection*{(1b) Tropicalizing the candidate solution.}\label{(1b)}\hfill

Recall from (\ref{eq:Psi}) that the candidate solution to (\ref{eq:NA-MA}) is 
\[\psi^\na(\nu) = u(\frac{e}{b_1} \nu(s_1), \cdots, \frac{e}{b_d}\nu(s_d)) \text{ \ on $X^\na$}.\]

To take the non-Archimedean Monge--Ampère measure of $\psi^\na$, we first need to verify that it belongs to $\mathrm{CPSH}(X^\na)$ in the sense of \cite{CLD25} (see Definition \ref{def:class-of-functions}). We show this in Lemma~\ref{lem:psi-na-psh}. Along the way, we will also see that $\psi^\na|_{U^\na}$ is tropicalized by $\tau$. Then, because $U$ is Zariski open in $X$, Equation~(\ref{cor:no-pluripolar-mass}) implies \[\NAMA(\psi^\na) = \NAMA(\psi^\na|_{U^\na}),\] 
which will allow us to work with $U^\na$ instead of $X^\na$ later. 

By Theorem \ref{thm:CTY}, $u$ is a convex differentiable function on the octant $\mathbf{R}^d_{\geq 0}$. For our purposes, we need $u$ to be defined on all of $\R^d$. But global continuous extension of a convex function is not always possible; see e.g. \cite{SS79}. Nonetheless, in this situation we can extend using homogeneity of $u$, carried out in Appendix~\hyperref[appen:A]{A}. Such an extension is not unique. 

\begin{proposition}\label{prop:extension}
    There is a continuously differentiable, non-negative and convex function $\tilde{u}$ on $\R^d$ extending $u$ that is identically zero on the half-space $\{t_1 + \cdots + t_d \leq 0\} \subseteq \mathbf{R}^d$, satisfies $\sum_{i=1}^d \frac{\partial \tilde{u}}{\partial t_i} \geq 0$, and $\tilde{u}(\lambda t_1, \cdots, \lambda t_d) = \lambda^{(n+d)/n} \tilde{u}(t_1, \cdots, t_d)$ for any $\lambda \in \mathbf{R}_{> 0}.$
\end{proposition}

\begin{convention}
    We use $u$ to denote the extension $\tilde{u}$ obtained in Proposition~\ref{prop:extension}.
\end{convention}

\begin{lemma}\label{lem:psi-na-psh}
    In the sense of \cite{CLD25} (see Definition~\ref{def:class-of-functions}), the function $\psi^\na$ lies in $\mathrm{CPSH}(X^\na)$ and so its restriction $\psi^\na|_{U^\na}$ lies in $\mathrm{CPSH}(U^\na)$. Thus both have well-defined non-Archimedean Monge--Ampère measures. 
\end{lemma}

\begin{proof}
    The proof is similar to that of Proposition~\ref{prop:fs-psh}. For any non-empty subset $J \subseteq [m]$, define an open subset of $X^\na$ by, \[U_J \coloneqq \bigcap_{j \in J} \{\nu(s_j') < 2\} \cap \bigcap_{j \not\in J} \{\nu(s_j') > 1\}.\] Because $\min_{0 \leq j \leq m} \nu(s_j') = 0$ for any $\nu \in X^\na$, the sets $\{U_J\}_J$ form an open cover of $X^\na$. 

    As the base field of analytification is trivially valued, there is a trivial metric $\phi_{\mathrm{triv}}$ on $eL^\na$, which by Example~\ref{def:triv-metric-over-triv-valued-field} can be written as \[\phi_{\mathrm{triv}} = \max_{0 \leq j \leq m} \log|s_j'|,\] since the chosen sections $s_0', \cdots, s_m'$ have no common zeros. Now on each $U_J$ we can write
    \[\frac{e}{b_i} \nu(s_i) = \phi_{\triv} - \log|s_i^{e/b_i}| = \max_{j \in J} \log|f_{ij}|\] where $f_{ij} \coloneqq s_j' / s_i^{e/b_i}$ are non-vanishing analytic functions on $U_J$. So on $U_J$ we have \begin{align*}\psi^\na(\nu) &= u(\frac{e}{b_1} \nu(s_1), \cdots, \frac{e}{b_d}\nu(s_d)) = \max_{j \in J} u(\log|f_{1j}|, \cdots, \log|f_{dj}|).\end{align*}
In the second equality we use the fact that $u$ increases diagonally, allowing us to pull out the maximum. It then follows from Examples~\ref{eg:comp-of-fn-is-psh} and \ref{eg:max-of-psh} that $\psi^\na|_{U_J}$ lies in $\mathrm{CPSH}(U_J)$. Since $\mathrm{CPSH}$ is a sheaf on $X^\na$, we have $\psi^\na \in \mathrm{CPSH}(X^\na)$, and consequently the restriction $\psi^\na|_{U^\na}$ is in $\mathrm{CPSH}(U^\na)$.
\end{proof}

This proof also allows us to see that the function $\psi^\na|_{U^\na} \colon U^\na \to \mathbf{R}$ is tropicalized by the moment map $\tau$ constructed in \hyperref[(1a)]{(1a)}. Indeed, let us define a function $v\colon \mathbf{R}^{N+1} \to \mathbf{R}$ by \begin{align*}\label{eq:v}v(x) \coloneqq u\big(\max_{0 \leq j \leq m} (x_{m+1} - x_j), \dots, \max_{0 \leq j \leq m} (x_{m+d} - x_j)\big).\end{align*} Observe that $v$ descends to $\mathbf{R}^{N+1} / \mathbf{R}$, where $\mathbf{R}$ acts additively; we continue to denote it by $v$. It follows from the proof of Lemma~\ref{lem:psi-na-psh} that \[(\tau^\trop)^* v = \psi^\na|_{U^\na}.\]

Finally, as mentioned before, by Equation~(\ref{cor:no-pluripolar-mass}), to prove (\ref{eq:NA-MA}) it suffices to show \[\NAMA(\psi^\na|_{U^\na}) = \mu^\na.\]

\subsection*{Step 2: Cellular decomposition of $\Sigma_{\tau}$}\label{step2}\hfill

We will freely use the definitions from \S \ref{sec:char-poly} concerning the characteristic polyhedron $\Sigma_\tau \subseteq X^\na$. In this step, we equip $\Sigma_{\tau}$ with an explicit cellular decomposition $\mathcal{C}$, which will allow us to compute the non-Archimedean Monge--Ampère measure later.

By the last step, we can work on $U$ where all $s_0', \cdots, s_m'$ are invertible, and we have a moment map $\tau\colon U^\na \to \mathbf{T}^\na$ where $\mathbf{T} = \mathbf{G}^{N+1}_m/ \mathbf{G}_m$ and $N = m+d$, induced by the collection of sections \[\mathcal{S} = \{s_0', \cdots, s_m', s_1^{e/b_1}, \cdots, s_d^{e/b_d}\}.\] For any subset $I \subseteq \mathcal{S}$ of size $n+1$, let $\tau_I$ denote the map $U^\na \to \mathbf{T}_I^\na$ induced by the sections in $I$, where $\mathbf{T}_I = \mathbf{G}^I_m/\mathbf{G}_m$. Note that $\tau_I$ factors through $\tau$ by projecting onto the $I$-coordinates. By Remark~\ref{rem:char-poly-by-proj}, the characteristic polyhedron is given by \[\Sigma_\tau = \bigcup_{|I| = {n+1}} \tau_I^{-1}(\mathrm{Sk}(\mathbf{T}_I)).\]

\begin{definition}\label{def:cover-of-char-poly}
    For any subset $J \subseteq \mathcal{S}$ with size $1 \leq |J| \leq n$, define a simple normal crossing divisor $F_J$ on $\bar{X}$ by $\sum_{s \in J} \mathrm{div}(s)_{\mathrm{red}}$. Using the construction in \hyperref[(1a)]{(1a)}, we obtain,
    \begin{enumerate}
        \item if $|J| < n$, then the deepest intersection stratum $Z_J$ of $F_J$ is irreducible. Define a cell \[C_J \coloneqq \QM_{Z_J} (\bar{X}, F_J) \subseteq \bar{X}^{\val} = X^\val;\]
        \item if $|J| = n$, the depth-$n$ intersection strata of $F_J$ is a finite collection of reduced points $\{p_{J,l}\}_l$. For each $l$ define \[C_{J,l} \coloneqq \QM_{p_{J,l}}(\bar{X}, F_J) \subseteq \bar{X}^{\val} = X^\val.\] The union $\bigcup_l C_{J,l}$ is $\QM(\bar{X}, F_J)$.
    \end{enumerate} 
    Let $\mathcal{C}$ be the collection of all these cells, and let $\mathcal{C}_i$ be the collection of $i$-dimensional cells. Observe that any $C_{J,l} \in \mathcal{C}_n$ is isomorphic to $\mathbf{R}^n_{\geq 0}$, and all cells are contained in $U^\na \subseteq X^\na$.
\end{definition}

We will first prove in Proposition~\ref{prop:cover} that $\mathcal{C}_n$ covers the characteristic polyhedron $\Sigma_\tau$, and then show in Proposition~\ref{prop:decomp} that $\mathcal{C}$ is a cellular decomposition of $\Sigma_\tau$.

\begin{proposition}\label{prop:cover}
    The union of cells in $\mathcal{C}_n$ equals the characteristic polyhedron $\Sigma_\tau$. 
\end{proposition}

\begin{proof}
    It suffices to show that, for any size $n+1$ subset $I$ of $\mathcal{S}$, there is an equality \begin{equation}\label{eq:cover}\tau_{I}^{-1} (\sk(\mathbf{T}_I)) = \bigcup_{J \subseteq I, |J| = n}\QM(\bar{X}, F_J).  \end{equation} By Example~\ref{eg:ske-torus}, the skeleton $\sk(\mathbf{T}_I)$ is $\QM(\mathbf{P}^n, V(z_0 \cdots z_n))$, where $\mathbf{T}_I = \mathbf{G}^I_m / \mathbf{G}_m$ embeds into $\mathbf{P}^I$ with homogeneous coordinates $z_0, \cdots, z_n$. Now the proposition follows directly from our choice of $\mathcal{S}$ and Lemma~\ref{lm:inverse-image-of-qm-val} below. \end{proof}
    
\begin{lemma}\label{lm:inverse-image-of-qm-val}
    Let $f\colon X \to X'$ be a surjective morphism of irreducible smooth projective varieties of the same dimension. Let $D'$ be a simple normal crossing divisor in $X'$ with $d$ irreducible components $D_1', \cdots, D_d'$ such that each $D_i  \coloneqq f^{-1}D_i'$ is an irreducible divisor on $X$ and $D \coloneqq D_1 + \cdots + D_d$ is also simple normal crossing, and moreover all positive-dimensional intersections among $D_1, \cdots, D_d$ are irreducible. Then $(f^\na)^{-1}(\QM(X', D')) = \QM(X, D).$
\end{lemma}

\begin{remark}
    The lemma is false if $D$ and $D'$ do not have same number of irreducible components. For example, consider the blowup morphism $f\colon \mathrm{Bl}_p \mathbf{P}^2 \to \mathbf{P}^2$ with an exceptional divisor $E$, and $D' = \{x = 0\} \subseteq \mathbf{P}^2$. If we set $D = \widetilde{D} + E$, where $\widetilde{D}$ is the strict transform of $D$ and $E$ is the exceptional divisor, then $f(D) = D'$. But $\QM(\mathrm{Bl}_p \mathbf{P}^2, D) \simeq \R^2_{\geq 0}$ whereas $\QM(\mathbf{P}^2, D') \simeq \R_{\geq 0}$.
\end{remark}

\begin{proof}
For each $1 \leq i \leq d$, we have $f^{-1} D_i' = a_i D_i$ for some $a_i \in \mathbf{Z}_{\geq 1}$. Therefore, \begin{equation}\label{eq:val-of-image}f^\na(\nu)(D_i') = a_i \nu(D_i)\text{ for any $\nu \in X^\na$.}\end{equation}

We begin by showing that, if $\nu \in \QM(X, D)$ then $\nu' \coloneqq f^\na(\nu) \in \QM(X', D')$. For this we use the tropical spectrum description of the Berkovich analytification in \S \ref{sec:trop-spec} and \S \ref{sec:qm}, where quasimonomial valuations on $D$ are characterized as those satisfying a certain minimality condition.

By the retraction map in \cite[\S 4.3]{JM12}, there is some $\chi' \in \QM(X', D')$ such that $\nu' \geq \chi'$ and $\nu'(D_i') = \chi'(D_i')$ for all $1 \leq i \leq d$, and $c_{X'}(\nu') \in \overline{c_{X'}(\chi')}$. By definition, the center $c_X(\nu)$ is the generic point of some intersection stratum $Z$ of $D$. Then $Z' \coloneqq f(Z)$ is an intersection stratum of $D'$, with $f(c_X(\nu)) = c_{X'}(f^\na(\nu)) = c_{X'}(\nu')$ being its generic point. So we have $c_{X'}(\nu') = c_{X'}(\chi')$. Because $f$ is surjective, so is $f^\na$. Hence $(f^\na)^{-1}(\chi')$ is non-empty. Say $Z = \overline{c_X(\nu)}$ is contained in the components $D_{i_1}, \cdots, D_{i_k}$ of $D$, and also decompose $D_{i_1} \cap \cdots \cap D_{i_k} = Z_1 \sqcup \cdots \sqcup Z_j$ where all $Z_p$ are irreducible of the same dimension, and $Z_1 = Z$. 

By \eqref{eq:val-of-image}, for any $\chi \in (f^\na)^{-1}(\chi')$, we have $\chi(D_i) = \nu(D_i)$ for all $1 \leq i \leq d$. Moreover, it follows from Remark~\ref{rem:center} that, for any subscheme $E$ of $X$, if $\chi(E) > 0$ then $c_X(\chi) \in E$. So $c_X(\chi) \in D_{i_1} \cap \cdots \cap D_{i_k}$, and $\overline{c_X(\chi)} = Z_p$ for some $1 \leq p \leq j$. 

If $\dim Z \geq 1$, then by our assumption on $D$ we have $j = 1$. That is, $c_X(\chi) = Z_1 = Z = c_X(\nu)$. So $\chi \geq \nu$ by minimality of $\nu$, and so $\chi' \geq \nu'$. Combined with the ordering $\nu' \geq \chi'$ from before, we obtain $\nu' = \chi' \in \QM(X', D')$, as desired. 

Therefore, we may assume all $Z_1, \cdots, Z_j$ are smooth points, and suppose $c_X(\chi) = Z_p$ for some $p \neq 1$. Then $\xi' \coloneqq c_{X'}(\chi') = c_{X'}(\nu')$ is also a smooth point in $X'$. Let $X^{\circ} \coloneqq X \backslash \{Z_2, \cdots, Z_j\}$. Then $f$ restricts to a map $f^\circ\colon X^\circ \to X'$, which is still surjective but satisfies $(f^{\circ})^{-1}(\xi') = Z$. So by surjectivity of $(f^\circ)^\na$, there is some $\chi_1$ centered at $Z_1 = Z$ with $f^\na(\chi_1) = \chi'$. By minimality of $\nu$ we have $\chi_1 \geq \nu$ again. So $f^\na(\chi_1) = \chi' \geq \nu'$. Hence we have $\nu'= \chi' \in \QM(X',D')$ as desired. 

We now show the other inclusion: if $\nu' \coloneqq f^\na(\nu) \in \QM(X',D')$, then $\nu \in \QM(X, D)$. Let $D_{i_1}, \cdots , D_{i_k}$ be the components of $D$ that contain $c_X(\nu)$. Then $0< \nu(D_{i_1}), \cdots, \nu(D_{i_k}) < \infty$, and so we can find a unique $\chi \in \QM(X,D)$ with $c_X(\nu) \in \overline{c_X(\chi)}$ and satisfying $\nu(D_{i_m}) = \chi(D_{i_m})$ for $1 \leq m \leq k$. In particular, $\chi \leq \nu$. It follows that $\chi' = f^\na(\chi) \leq \nu'$, and \[c_{X'}(\nu') = f(c_X(\nu)) \in f(\overline{c_X(\chi)}) \subseteq \overline{f(c_X(\chi))} = \overline{c_{X'}(\chi')}.\] In addition, by \eqref{eq:val-of-image} we have $\chi'(D_{i_1}'), \cdots, \chi'(D_{i_k}') > 0$, so $c_{X'}(\chi') \in D_{i_1}' \cap \cdots \cap D_{i_k}'$. This implies $c_{X'}(\chi') \in \overline{c_{X'}(\nu')}$, and so $c_{X'}(\chi') = c_{X'}(\nu')$. By \eqref{eq:val-of-image} again, $\nu'(D_i) = \chi'(D_i)$ for all $i$, so by minimality of $\nu'$ we have $\nu' \leq \chi'$. We have shown before $\chi' \leq \nu'$, so $f^\na(\chi) = \chi' = \nu'\in \QM(X', D')$. Now by the other inclusion proven earlier, as $\chi \in \QM(X,D)$, we must have $\nu \in \QM(X',D')$.  \end{proof}

We have a simple corollary. As before, the skeleton $\mathrm{Sk}(\mathbf{T}_I)$ is the union of $n+1$ simplices $Q_i \coloneqq \QM_{e_i}(\mathbf{P}^n, V(z_0 \cdots z_n))$, where $\mathbf{T}_I \simeq \mathbf{G}^I_m / \mathbf{G}_m$ embeds into $\mathbf{P}^I = \mathbf{P}^n$ with homogeneous coordinates $z_0, \cdots z_n$, and $e_i$ is the point whose $i$-th coordinate is $1$ and all others are $0$. 

\begin{cor}\label{cor:cover}
    With the notation above, for any size $n+1$ subset $I$ of $\mathcal{S}$ and any $0 \leq i \leq n$, the set $(\tau_I^\trop)^{-1}(Q_i)$ is a finite union of simplicial cones $\QM_p(\bar{X}, F_J)$, where $p$ ranges over a subset of the depth-$n$ intersection strata of $F_J$ for $J \subseteq I$ with $|J| = n$. In addition, $\tau_I$ maps any such cone homeomorphically onto $Q_i$. So $\tau_I$ is a finite covering over the interior of $Q_i$. 
\end{cor}

\begin{proof}
    By construction, for any $J \subseteq I$ of size $n$, the divisor $F_J$ is a finite union of smooth points $\{p_{J, l}\}_l$. By the forward inclusion of (\ref{eq:cover}) we have \[\tau_I^{-1}(Q_i) \subseteq \bigcup_{J \subseteq I, |J| = n} \QM(\bar{X}, F_J) = \bigcup_{p \in \{p_{J,l}\}_{J,l}} \QM_p(\bar{X}, F_J).\] On the other hand, from the backward inclusion of (\ref{eq:cover}), for any $p \in \{p_{J,l}\}_{J,l}$, the map $\tau_I$ sends $\QM_p(\bar{X}, F_J)$ homeomorphically onto $Q_i$ for some $i$. So $\tau^{-1}_I(Q_i)$ is a union of $\QM_p(\bar{X}, F_J)$ for a subset of $\{(p,J): p \in \{p_{J,l}\}_{J,l}, J \subseteq I, |J| = n\}$, as desired. 
\end{proof}

\begin{proposition}\label{prop:decomp}
    The collection $\mathcal{C}$ is a cellular decomposition of $\Sigma_\tau$ in the sense of Definition~\ref{def:cell-decomp}. 
\end{proposition}

\begin{proof}
    Condition (\ref{dcp-i}) in Definition~\ref{def:cell-decomp} follows from Proposition~\ref{prop:cover} and Corollary~\ref{cor:cover}. 
    It is a straightforward combinatorial exercise to check (\ref{dcp-ii}) and (\ref{dcp-iii}). Indeed, for any two cells $C_I \neq C_J \in \mathcal{C}$, if $|I|, |J| < n$, then $C_I \cap C_J$ equals $C_{I \cap J}$. If $|I| = n$ and $|J| < n$, then $C_{I,l} \cap C_J = C_{I \cap J}$. If $|I|, |J| = n$ and $I \neq J$, then $C_{I,l} \cap C_{J,l'} = \bigcup_{I' \subseteq I \cap J,\ 1 \leq |I'| \leq n-1} C_{I'}$. This shows (\ref{dcp-ii}). The proof for (\ref{dcp-iii}) is similar, and (\ref{dcp-iv}) is clear. 
\end{proof}

Recall from \S \ref{sec:essential-ske} that the essential skeleton $\sk(X)$ is $\QM(\bar{X}, D) \subseteq U^\na \subseteq X^\na$. Observe that $\sk(X)$ is contained in $\Sigma_\tau$ by Lemma~\ref{lm:inverse-image-of-qm-val}. Let $\Delta$ denote the image $\tau^\trop(\sk(X)) \subseteq \NRT = \R^{N+1} / \mathbf{R}$. It can be described explicitly as
\[\Delta = \{x_0 = \cdots = x_m \leq x_{m+1}, \cdots, x_{m+d}\}/ \mathbf{R},\] where the $x_i$ are homogeneous coordinates on $\mathbf{R}^{N+1}$.

Finally, with this cellular decomposition in hand, we specify the constant $c$ appearing in Theorem~\ref{thm:CTY} as \begin{equation}\label{constant}
    c = \frac{(n-d)!}{\#\{C \in \mathcal{C}_n: \sk(X) \subseteq C\}}.
\end{equation}

\subsection*{Step 3: Rewriting the NA MA equation}\label{step3} \hfill

In this short step, we use the cellular decomposition constructed in Step \hyperref[step2]{2} to rewrite the non-Archimedean Monge--Ampère equation \begin{equation}\label{eq:NA-MA-2}\NAMA(\psi^\na|_{U^\na}) = \mu^\na \text{  on $U^\na$}\tag{\ref{eq:NA-MA}}\end{equation} as a real Monge--Ampère equation, where $\psi|_{U^\na}$ is the candidate solution tropicalized by the moment map $\tau$ constructed in Step \hyperref[step1]{1}, $U$ is the Zariski open subset $\{s_0', \cdots, s_m' \neq 0\}$ of $X$, and $\mu^\na$ is the Lebesgue measure on the essential skeleton $\sk(X) \subseteq U^\na \subseteq X^\na$. 

Let $v_k\colon \NRT \to \mathbf{R}$ be smooth convex functions converging to $v$ as $k \to \infty$, and set $\psi_k \coloneqq (\tau^\trop)^* v_k$. For any cell $C \in \mathcal{C}_n$, the image $\sigma = \tau^\trop(C)$ is isomorphic to $\mathbf{R}^n_{\geq 0}$ by Definition~\ref{def:cover-of-char-poly}. Let $\mathbf{L}_\sigma$ denote the $n$-dimensional affine subspace of $\NRT$ spanned by $\sigma$. Because $v_k$ is smooth, by Definition~\ref{def:measure-on-each-cell} the non-Archimedean Monge--Ampère measure of $\psi_k$ is \[\NAMA(\psi_k)  = \sum_{\sigma} \mu_{\sigma, v_k},\] where the sum runs over $\sigma = \tau^\trop(C)$ for all $C \in \mathcal{C}_n$ and \[\mu_{\sigma, v_k} = (\tau^\trop|_C^{-1})_* \mathbf{1}_{\sigma} \RMA(v_k|_{\mathbf{L}_{\sigma}}).\]

By construction $v$ is not $C^2$. By continuity of the non-Archimedean Monge--Ampère operator, the measure $\NAMA(\psi^\na)$ is the weak limit of $\NAMA(\psi_k)$. Hence, the equation (\ref{eq:NA-MA}) is equivalent to the real Monge--Ampère equation,
\begin{equation}\label{eq:NA-MA-update}\lim_{k \to \infty} \sum_{\sigma} \mathbf{1}_\sigma \RMA(v_k|_{\mathbf{L}_\sigma}) = \mathrm{Leb}_{\Delta},\end{equation} where $\mathrm{Leb}_{\Delta}$ is the Lebesgue measure on $\Delta = \tau^\trop(\sk(X))$,  which exists because $\NRT$ has a natural integral affine structure. 

As already discussed in \S \ref{sec:comparsion}, a difficulty in solving \eqref{eq:NA-MA-update} is the presence of the indicator function $\mathbf{1}_\sigma$. Indeed, because $\sigma$ is closed in $\mathbf{L}_\sigma$, it is hard to pass from weak convergence of Monge--Ampère measures on $\mathbf{L}_\sigma$ to that on $\sigma$. We will deal with this problem in Step \hyperref[step5]{5}, and in Step \hyperref[step4]{4} establish a simpler equation to be used later. 
\subsection*{Step 4: A simpler equation via a volume calculation}\label{step4}\hfill

Instead of (\ref{eq:NA-MA-update}), the goal of this step is to prove a simpler relation,\begin{equation}\label{step4:main-eq}
    \sum_{C \in \mathcal{C}_n} \RMA(v|_{\mathbf{L}_{\sigma}}) = \Leb_{\Delta}\end{equation}
where $\sigma = \tau^\trop(C)$ and $\mathbf{L}_\sigma$ is the $n$-dimensional affine space spanned by $\sigma$ in $\NRT = \mathbf{R}^{N+1} / \mathbf{R}$. The operator $\RMA(-)$ is the Alexandrov Monge--Ampère operator, and $v\colon \NRT \to \mathbf{R}$ is the convex function defined as \[v = \max_{0 \leq j \leq m} u(x_{m+1} - x_j, \cdots, x_{m+d} - x_j).\] We will use (\ref{step4:main-eq}) to prove (\ref{eq:NA-MA-update}) in Step \hyperref[step5]{5} after constructing a suitable smoothing of $v$.

The key in this step is to interpret (\ref{step4:main-eq}) as a calculation of the volume of certain subgradient polytopes, where the maximum expression in the trivial metric will play an important role. In fact, we will prove a more precise result which implies (\ref{step4:main-eq}).

\begin{proposition}\label{prop:step4-cases}
    Let $c$ be the positive constant specified in (\ref{constant}). We have \[\RMA(v|_{\mathbf{L}_\sigma}) = \begin{cases}
        c \cdot (n-d)!^{-1} \cdot \mathrm{Leb}_{\Delta} \, & \text{ if } \Delta \subseteq \sigma,\\
        0 & \text{ otherwise}.
    \end{cases} \]
\end{proposition}

\begin{convention}\label{convention:step4}\hfill
\begin{itemize}[leftmargin=*]
    \item To ease notation, we fix any size $n$ subset $J$ of the set $\mathcal{S}$ of sections constructed in Step~\hyperref[step1]{1}, and let $C$ be the $n$-dimensional cell $C_{J,l} \in \mathcal{C}_n$ for any $l$ (see Definition~\ref{def:cover-of-char-poly}). We use $\mathbf{L}$ to denote $\mathbf{L}_{\sigma}$, the $n$-dimensional affine space spanned by $\sigma = \tau^\trop(C)$.  
    \item Since $m \geq n$, there is some $0 \leq j \leq m$ such that $\chi(s_j') = 0$ for any $\chi \in C$. Note that $s_j' \not\in J$. Let $I$ be $J \cup \{s_j'\}$, and let $\mathbf{R}^{N+1} / \mathbf{R} \twoheadrightarrow \mathbf{R}^I / \mathbf{R}$ be the projection map. The composition \[\mathbf{L} \hookrightarrow \NRT = \mathbf{R}^{N+1} / \mathbf{R} \twoheadrightarrow \mathbf{R}^I / \mathbf{R}\] is an isomorphism, fixed from now on. 
    \item Let $x_0, \cdots, x_N$ be homogeneous coordinates on $\mathbf{R}^{N+1}$. To avoid dealing with the equivalence relation on $\NRT = \mathbf{R}^{N+1} / \mathbf{R}$, for any $0 \leq i \neq j \leq N$, we define an affine function $X_i \coloneqq x_i - x_j$ on $\NRT$. Then $(X_i)_i$ are affine coordinates on $\NRT$, and $(X_i)_{i \in J}$ are affine coordinates of $\mathbf{L} \simeq \mathbf{R}^I / \mathbf{R}$. Note that if $i \not\in J$, then $X_i = 0$ on $\mathbf{L}$. 
    \item In these coordinates, we can write \begin{equation}\label{eq:delta-in-coord}\Delta = \tau^\trop(\sk(X)) = \{X_0 = \cdots = X_m = 0 \leq X_{m+1}, \cdots, X_{m+d}\} \subseteq \NRT.\end{equation}
    \item After permuting, we can write $J = \{0, \cdots, n-r-1, m+1, m+r\}$ for some fixed number $r \in [0,d]$. Then $X_0, \cdots, X_{n-r-1}, X_{m+1}, \cdots, X_{m+r}$ are coordinates on $\mathbf{L}$, and the gradient operator $\nabla$ on $\mathbf{L}$ is taken with respect this choice of coordinates.
\end{itemize}
\end{convention}

\noindent For each $0 \leq k \leq n-r-1$, define an affine map $L_k\colon\mathbf{L} \to \mathbf{R}^d$ by, \[L_k \coloneqq (\underbrace{X_{m+1} - X_k, \cdots, X_{m+r} - X_k}_r, \underbrace{-X_k, \cdots, -X_k}_{d-r}),\] and define another affine map $L \colon \mathbf{L} \to \mathbf{R}^d$ by, \[L \coloneqq (\underbrace{X_{m+1}, \cdots, X_{m+r}}_r, \underbrace{0, \cdots, 0}_{d-r}).\] Then, let $u\colon \mathbf{R}^d \to \mathbf{R}$ be the convex function from Proposition~\ref{prop:extension}, which extends the PDE solution in Theorem~\ref{thm:CTY} and is differentiable. Define differentiable functions $U_k, U\colon \mathbf{L} \to \mathbf{R}$ for $0 \leq k \leq n-r-1$ by \[U_k \coloneqq u \circ L_k\text{ and }U \coloneqq u \circ L.\] With this, we can write the restriction of $v$ to $\mathbf{L}$ simply as \begin{align*}
    v|_{\mathbf{L}}(x) =& \max_{0 \leq k \leq m} u(x_{m+1} - x_k, \cdots, x_{m+d} - x_k)|_{\mathbf{L}}\\
    =&\max_{0 \leq k \leq n-r-1} \{U_k, U\}.
\end{align*}
Fix any point $p \in \mathbf{L}$. To prove Proposition~\ref{prop:step4-cases}, we need to compute the subgradient polytope of $v|_{\mathbf{L}}$ at $p$. For a function $f = \max\{f_1, \ldots, f_k\}$ where each $f_i$ is convex, we say $f_i$ is \emph{active} at $p$ if $f_i(p) = f(p)$. If each $f_i$ is differentiable, the subgradient polytope of $f$ at $p$, denoted $\partial_p f$, is the convex hull of $\nabla f_i(p)$ for those $f_i$ active at $p$. 

So we compute the gradient of each term participating in the maximum that defines $v|_{\mathbf{L}}$. Let $\partial_i u$ denote the $i$-th derivative of $u \colon \R^d \to \R$ for each $1 \leq i \leq d$. The gradient image of any function on $\mathbf{L}$ lies naturally in the dual space $\mathbf{L}^\vee$ with basis \[X_0^\vee, \cdots, X_{n-r-1}^\vee, X_{m+1}^\vee, \cdots, X_{m+r}^\vee.\]

Since $U = u \circ L$ does not depend on $X_0, \cdots, X_{n-r-1}$, the chain rule gives 
\begin{equation}\label{eq:gradient-tilde-U}\nabla U(p) = (\underbrace{0, \cdots, 0}_{n-r}, \underbrace{(\partial_1 u)|_{L(p)}, \cdots, (\partial_r u)|_{L(p)}}_r) \in \mathbf{L}^\vee.\end{equation}
Similarly, for any $0 \leq k \leq n-r-1$, since $U_k = u \circ L_k$, the chain rule gives \begin{equation}\label{eq:gradient-U_i}\nabla U_k(p) = (\underbrace{0, \cdots, 0}_{k-1}, -\sum_{i=1}^d (\partial_i u)|_{L_k(p)}, \underbrace{0, \cdots, 0}_{n-k-r}, \underbrace{(\partial_1 u)|_{L_k(p)}, \cdots, (\partial_r u)|_{L_k(p)}}_{r}) \in \mathbf{L}^\vee.\end{equation} 

We analyze by cases depending on the coordinates of $p$. 

\subsubsection*{Case 1: There is some $0 \leq k \leq n-r-1$ such that $U_k$ is not active at $p$.}
\label{s4c1}\hfill

We can choose an open neighborhood $R \subseteq \mathbf{L}$ of $p$ such that $U_k$ is not active at any point on $R$. For any $q \in R$, by the calculations (\ref{eq:gradient-tilde-U}) and (\ref{eq:gradient-U_i}), the vertices of the gradient polytope $\partial_q (v|_{\mathbf{L}})$ lie in the hyperplane $\{X_k^\vee = 0\}$ of $\mathbf{L}^\vee$, and hence the entire gradient polytope is contained in this hyperplane as well. It follows that \[\RMA(v|_{\mathbf{L}})(R) = \mathrm{vol}_{\mathbf{L}^\vee}(\bigcup_{q \in R} \partial_q(v|_{\mathbf{L}})) = 0.\]

\subsubsection*{Case 2: $U$ is not active at $p$.}\label{s4c2}\hfill

This is similar to Case 1. We can find an open neighborhood $R \subseteq \mathbf{L}$ of $p$ such that for any point $q \in R$, the subgradient polytope $\partial_q(v|_{\mathbf{L}})$ is contained in the hyperplane $\{\sum_{0 \leq k \leq n-r-1} X_k + \sum_{1 \leq k' \leq r} X_{m+k'} = 0\}.$ So $\RMA(v|_{\mathbf{L}})(R) = 0$. \\

By the two cases above, the support of $\RMA(v|_{\mathbf{L}})$ is contained in the set of points $p$ where $U_0, \cdots, U_{n-r-1}, U$ are all active in $v|_{\mathbf{L}}$; that is, for every $0 \leq k \leq n-r-1$ we have $U_k(p) = U(p)$, which forces $X_k(p) = 0$. Because $X_{n-r+1} = \cdots = X_m = 0$ on $\mathbf{L}$ as well, we must have $p \in \Delta$ by \eqref{eq:delta-in-coord}. Thus, 

\begin{cor}\label{cor:supp-RMA}
    The support of $\RMA(v|_{\mathbf{L}})$ is contained in $\mathbf{L} \cap \Delta$. 
\end{cor}

In the remaining cases, we assume $U, U_0, \cdots, U_{n-r-1}$ are all active at $p$. 
\subsubsection*{Case 3: All $U_k, U$ are active at $p$ and $\mathbf{L}$ does not contain $\Delta$}\label{s4c3} \hfill

Because $\mathbf{L}$ does not contain $\Delta$, the set of sections $I$ does not contain all of $\{s_1^{e/b_1}, \cdots, s_d^{e/b_d}\}$; so $r < d$. Thus $\mathbf{L}$ can only intersect $\Delta$ along $\partial \Delta$; that is, $\mathbf{L} \cap \Delta = \mathbf{L} \cap \partial \Delta$. So $p \in \partial \Delta$. We claim $\RMA(v|_{\mathbf{L}})$ is the zero measure. For this we need to use the boundary condition in Theorem~\ref{thm:CTY}: \begin{equation}\label{eq:BC-2}\sum_{i=1}^d \partial_i u = 0 \text{ on }\partial \mathbf{R}^d_{\geq 0}.\end{equation}

By Corollary~\ref{cor:supp-RMA}, it suffices to prove that $\RMA(v|_{\mathbf{L}})(\mathbf{L} \cap \partial \Delta) = 0.$ To this end, we use (\ref{eq:gradient-tilde-U}) and (\ref{eq:gradient-U_i}) again. Because $p \in \mathbf{L} \cap \partial \Delta$, observe that $L(p), L_k(p) \in \partial \R^d_{\geq 0} \subseteq \R^d$ for each $0 \leq k \leq n-r-1$. Plugging (\ref{eq:BC-2}) into the $i$-th coordinate of (\ref{eq:gradient-U_i}), we see that subgradient polytope $\partial_p (v|_{\mathbf{L}})$ is contained in the subspace $\{X_0^\vee = \cdots = X_{n-r-1}^\vee = 0\} \subsetneq \mathbf{L}^\vee$. It follows that \[\RMA(v|_{\mathbf{L}})(\mathbf{L} \cap \partial \Delta) = \mathrm{vol}_{\mathbf{L}^\vee} (\bigcup_{p \in \mathbf{L} \cap \partial \Delta} \partial_p(v|_{\mathbf{L}})) = 0.\]

\begin{remark}
    From this, one can see that (\ref{eq:BC-2}) is the most straightforward boundary condition to put on $u$ in order to prevent $\RMA(v|_{\mathbf{L}})$ from putting singular charges on $\partial \Delta$, or equivalently on $\partial \sk(X)$. 
\end{remark}

\subsubsection*{Case 4: All $U_k, U$ are active at $p$ and $\mathbf{L}$ contains $\Delta$} \hfill

So $r = d$, or equivalently $s_1^{e/b_1}, \cdots, s_d^{e/b_d} \in I$. As in the proof of Case \hyperref[s4c3]{3}, using the boundary condition from Theorem~\ref{Conv:TY-metric} we have $\RMA(v|_{\mathbf{L}}) (\partial \Delta) = 0$. Finally, we prove the following lemma via a volume calculation, relying on the PDE from Theorem~\ref{thm:CTY}.

\begin{lemma}\label{lem:mu_ac}
    We have $\RMA(v|_{\mathbf{L}}) = c \cdot (n-d)!^{-1} \cdot \mathrm{Leb}_{\Delta}$, where $\mathrm{Leb}_\Delta$ is the Lebesgue measure on $\Delta \cap \mathbf{L} = \Delta$ and $c$ is the constant from (\ref{constant}). 
\end{lemma}

\begin{proof}
    By definition $\Delta = \tau^\trop(\sk(X))$ and $X_{m+1}, \cdots, X_{m+d}$ form coordinates on $\Delta$. Consider any open box $R \coloneqq \prod_{i=1}^d \{a_i < X_{m+i} < b_i\}$ in $\mathrm{relint}(\Delta)$, where $0 < a_i < b_i$. 

    In (\ref{eq:gradient-tilde-U}) and (\ref{eq:gradient-U_i}), for any point $p \in \Delta$ and any $0 \leq k \leq n - d - 1$ we have \[L_k(p) = L(p) = (- X_{m+1}(p), \cdots, - X_{m+d}(p)) \in \R^d.\] Temporarily, write $\partial_i u$ for the value $(\partial_i u)|_{L(p)}$ for brevity. Since $U$ and all $U_0, \cdots, U_{n-d-1}$ are active at $p$, the subgradient polytope $\partial_p (v|_{\mathbf{L}})$ is the convex hull of $n-d+1$ points in $\mathbf{L}^\vee$, which can be read off as the rows in the following $(n-d+1,n)$ matrix, where all the vacant entries are zero.
    \begin{figure}[h]\[
\left.
\begin{bmatrix}\,
\smash{
  \begin{matrix}
    0 &        &        &        & \partial_1 u & \cdots & \partial_d u \\
    -\sum_{i=1}^d \partial_i u &  &        &        & \partial_1 u & \cdots & \partial_d u \\
           & \ddots &    &        & \vdots     &   & \vdots \\
           &        & -\sum_{i=1}^d \partial_i u &  & \partial_1 u & \cdots & \partial_d u
  \end{matrix}
}
\vphantom{
  \begin{matrix}
  \smash[b]{\vphantom{\Big|}}\\
  \smash[b]{\vphantom{\Big|}}\\
  \vdots\\
  \smash[t]{\vphantom{\Big|}}
  \end{matrix}
}
\,\end{bmatrix}
\right\rbrace{\scriptstyle n - d + 1}
\]

\vspace{-1.2em}
\[
\hspace{-3.5em}
\underbrace{\hspace{13em}}_{n-d}
\hspace{0.5em}
\underbrace{\hspace{6.5em}}_{d}
\]\end{figure}

When evaluated at $L(p)$, the last $d$ coordinates of all the $n-d+1$ points are the same. Thus the polytope $\partial_p (v|_{\mathbf{L}})$ is isomorphic to the standard $n-d$ simplex with one vertex at the origin, where each edge emanating from the origin has length \[|-\sum_{i=1}^d (\partial_i u)|_{L(p)}|  = \sum_{i=1}^d (\partial_i u)|_{L(p)},\] which is positive because $u$ grows diagonally. The standard $n-d$ simplex has volume $(n-d)!^{-1}$, and so \[\mathrm{vol}_{\mathbf{R}^{n-d}} (\partial_p (v|_{\mathbf{L}})) = (n-d)!^{-1}(\sum_{i=1}^d (\partial_i u)|_{L(p)})^{n-d}.\]

Also observe from the matrix that the gradient functions $\partial_1 u, \cdots, \partial_d u$ are the dual coordinates to $X_{m+1}, \cdots, X_{m+d}$. The latter are exactly the coordinates on the $d$-dimensional box $R \subseteq \Delta$. Therefore, \begin{align*}\RMA(v|_{\mathbf{L}})(R) &= \mathrm{vol}_{\mathbf{L}^\vee}(\bigcup_{p \in R} \partial_p (v|_{\mathbf{L}}))\\ &= \int_R \mathrm{vol}_{\R^{n-d}} (\partial_p (v|_{\mathbf{L}})) d(\partial_1 u) \wedge \cdots \wedge d(\partial_d u)\\
    &= (n-d)!^{-1} \int_R (\sum^d_{i=1} \partial_i u)^{n-d} d(\partial_1 u) \wedge \cdots \wedge d(\partial_d u),\\
    &= (n-d)!^{-1} \int_R (\sum^d_{i=1} \partial_i u)^{n-d} \det(D^2 u) d\, \mathrm{Leb}_\Delta\\
    &= c \cdot (n-d)!^{-1} \cdot \int_R d\, \mathrm{Leb}_\Delta, \end{align*}
    where the last equality uses Theorem~\ref{thm:CTY}. So $\RMA(v|_{\mathbf{L}}) = c \cdot (n-d)!^{-1} \cdot \Leb_\Delta$, as desired. 
\end{proof}
\subsubsection*{Concluding the case analysis}\hfill 
\\
We have shown Proposition~\ref{prop:step4-cases}, which states
\begin{equation}\label{eq:step4-cases}\RMA(v|_{\mathbf{L}_{\sigma}}) = \begin{cases}
   c \cdot (n-d)!^{-1} \cdot \Leb_\Delta & \text{ if } \Delta\subseteq \sigma \\[2pt]
    0 & \text{ otherwise.}
\end{cases}\end{equation}

The case analysis leading to (\ref{eq:step4-cases}) is summarized in Figure \ref{fig:cases}. 

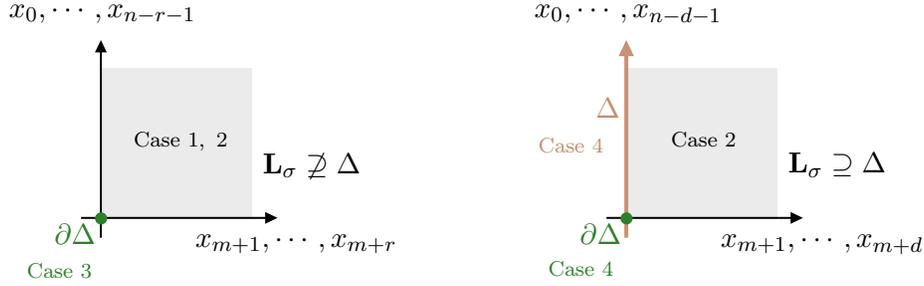
\begin{figure}[h]
\tikzset{every picture/.style={line width=0.85pt}} 

\begin{tikzpicture}[x=0.75pt,y=0.75pt,yscale=-1,xscale=1][ht]

\draw  [draw opacity=0][fill={rgb, 255:red, 235; green, 235; blue, 235 }  ,fill opacity=1 ] (154,54.96) -- (230.04,54.96) -- (230.04,131) -- (154,131) -- cycle ;
\draw [color={rgb, 255:red, 0; green, 0; blue, 0 }  ,draw opacity=1 ][line width=0.75]    (153.89,43.96) -- (153.89,140.74) ;
\draw [shift={(153.89,40.96)}, rotate = 90] [fill={rgb, 255:red, 0; green, 0; blue, 0 }  ,fill opacity=1 ][line width=0.08]  [draw opacity=0] (6.25,-3) -- (0,0) -- (6.25,3) -- cycle    ;
\draw [color={rgb, 255:red, 0; green, 0; blue, 0 }  ,draw opacity=1 ][line width=0.75]    (143.89,130.74) -- (240,130.74) ;
\draw [shift={(243,130.74)}, rotate = 180] [fill={rgb, 255:red, 0; green, 0; blue, 0 }  ,fill opacity=1 ][line width=0.08]  [draw opacity=0] (6.25,-3) -- (0,0) -- (6.25,3) -- cycle    ;
\draw  [color={rgb, 255:red, 46; green, 126; blue, 42 }  ,draw opacity=1 ][fill={rgb, 255:red, 46; green, 126; blue, 42 }  ,fill opacity=1 ] (151.52,131) .. controls (151.52,129.63) and (152.63,128.52) .. (154,128.52) .. controls (155.37,128.52) and (156.48,129.63) .. (156.48,131) .. controls (156.48,132.37) and (155.37,133.48) .. (154,133.48) .. controls (152.63,133.48) and (151.52,132.37) .. (151.52,131) -- cycle ;
\draw  [draw opacity=0][fill={rgb, 255:red, 235; green, 235; blue, 235 }  ,fill opacity=1 ] (419,54.96) -- (495.04,54.96) -- (495.04,131) -- (419,131) -- cycle ;
\draw [color={rgb, 255:red, 201; green, 145; blue, 118 }  ,draw opacity=1 ][fill={rgb, 255:red, 184; green, 233; blue, 134 }  ,fill opacity=1 ][line width=1.5]    (418.89,44.96) -- (418.89,140.74) ;
\draw [shift={(418.89,40.96)}, rotate = 90] [fill={rgb, 255:red, 201; green, 145; blue, 118 }  ,fill opacity=1 ][line width=0.08]  [draw opacity=0] (8.13,-3.9) -- (0,0) -- (8.13,3.9) -- cycle    ;
\draw [color={rgb, 255:red, 0; green, 0; blue, 0 }  ,draw opacity=1 ][line width=0.75]    (408.89,130.74) -- (505,130.74) ;
\draw [shift={(508,130.74)}, rotate = 180] [fill={rgb, 255:red, 0; green, 0; blue, 0 }  ,fill opacity=1 ][line width=0.08]  [draw opacity=0] (6.25,-3) -- (0,0) -- (6.25,3) -- cycle    ;
\draw  [color={rgb, 255:red, 46; green, 126; blue, 42 }  ,draw opacity=1 ][fill={rgb, 255:red, 46; green, 126; blue, 42 }  ,fill opacity=1 ] (416.52,131) .. controls (416.52,129.63) and (417.63,128.52) .. (419,128.52) .. controls (420.37,128.52) and (421.48,129.63) .. (421.48,131) .. controls (421.48,132.37) and (420.37,133.48) .. (419,133.48) .. controls (417.63,133.48) and (416.52,132.37) .. (416.52,131) -- cycle ;

\draw (169,86) node [anchor=north west][inner sep=0.75pt]  [font=\scriptsize] [align=left] {Case $\displaystyle 1,\ 2$};
\draw (234,96.4) node [anchor=north west][inner sep=0.75pt]    {$\mathbf{L}_{\sigma} \not\supseteq \Delta$};
\draw (106,21.4) node [anchor=north west][inner sep=0.75pt]    {$x_{0} ,\cdots,x_{n-r-1}$};
\draw (200,137.4) node [anchor=north west][inner sep=0.75pt]    {$x_{m+1} ,\cdots, x_{m+r}$};
\draw (129,132.4) node [anchor=north west][inner sep=0.75pt]  [color={rgb, 255:red, 46; green, 126; blue, 42 }  ,opacity=1 ]  {$\partial \Delta $};
\draw (115,152) node [anchor=north west][inner sep=0.75pt]  [font=\scriptsize,color={rgb, 255:red, 46; green, 126; blue, 42 }  ,opacity=1 ] [align=left] {Case 3};
\draw (440,86) node [anchor=north west][inner sep=0.75pt]  [font=\scriptsize] [align=left] {Case $\displaystyle 2$};
\draw (499,96.4) node [anchor=north west][inner sep=0.75pt]    {$\mathbf{L}_{\sigma} \supseteq \Delta$};
\draw (371,21.4) node [anchor=north west][inner sep=0.75pt]    {$x_{0} ,\cdots ,x_{n-d-1}$};
\draw (465,137.4) node [anchor=north west][inner sep=0.75pt]    {$x_{m+1} ,\cdots, x_{m+d}$};
\draw (394,132.4) node [anchor=north west][inner sep=0.75pt]  [color={rgb, 255:red, 46; green, 126; blue, 42 }  ,opacity=1 ]  {$\partial \Delta $};
\draw (378,151) node [anchor=north west][inner sep=0.75pt]  [font=\scriptsize,color={rgb, 255:red, 46; green, 126; blue, 42 }  ,opacity=1 ] [align=left] {Case 4};
\draw (402,68.4) node [anchor=north west][inner sep=0.75pt]  [color={rgb, 255:red, 196; green, 141; blue, 108 }  ,opacity=1 ]  {$\Delta$};
\draw (373,89) node [anchor=north west][inner sep=0.75pt]  [font=\scriptsize,color={rgb, 255:red, 196; green, 141; blue, 108 }  ,opacity=1 ] [align=left] {Case 4};

\raggedbottom
\end{tikzpicture}

\caption{\label{fig:cases} The left illustrates $\mathbf{L}_{\sigma}$ that does not contain $\Delta$. From the case analysis, we have $\RMA(v|_{\mathbf{L}_{\sigma}}) = 0$. The right illustrates $\mathbf{L}_{\sigma}$ that contains $\Delta$, where we have shown $\RMA(v|_{\mathbf{L}_{\sigma}}) = c \cdot (n-d)!^{-1}\cdot  \Leb_\Delta$.}\end{figure}~\\ \hfill
Lastly, through this case analysis, we arrive at the desired relation~\eqref{step4:main-eq}.

\begin{cor}
Equation~\eqref{step4:main-eq} holds. That is, \[\sum_{\sigma} \RMA(v|_{\mathbf{L}_{\sigma}}) = \mathrm{Leb}_{\Delta},\] where the sum runs over $\sigma = \tau^\trop(C)$ for all $C \in \mathcal{C}_n$. 
\end{cor}

\begin{proof}
    Let $a$ be the number of cells $C \in \mathcal{C}_n$ containing $\sk(X)$, or equivalently the number of $\sigma = \tau^\trop(C)$ containing $\Delta$. By the choice of the constant $c$ in \eqref{constant}, we have $c \cdot (n-d)!^{-1} = a^{-1}$. It then follows from Proposition~\ref{prop:step4-cases} that 
    \[\sum_{\sigma} \RMA(v|_{\mathbf{L}_{\sigma}}) = \sum_{\sigma \colon \Delta \subseteq \sigma} a^{-1} \Leb_\Delta = \Leb_\Delta,\] as desired.
\end{proof}

\subsection*{Step 5: Approximations}\label{step5} \hfill

Let us summarize what we have so far. Say $v_k$ are smooth convex functions that converge to $v$. For convenience, let $\mu_{\sigma,k}$ and $\mu_\sigma$ denote $\RMA(v_k|_{\mathbf{L}_{\sigma}})$ and $\RMA(v|_{\mathbf{L}_{\sigma}})$, respectively. As $k \to \infty$ we have weak convergence of measures, 
\begin{equation}\label{eq:step5-weak-conv-before-res}\sum_{\sigma} \mu_{\sigma,k} \to \sum_{\sigma} \mu_{\sigma} \stepeq \Leb_{\Delta},\end{equation} where the sum runs over $\sigma = \tau^\trop(C)$ for all $C \in \mathcal{C}_n.$ 

Also recall that our potential solution to (\ref{eq:NA-MA}) is $\psi^\na = (\tau^\trop)^* v$. By Step \hyperref[step3]{3}, showing $\NAMA(\psi^\na) = \mu^\na$ is equivalent to showing the following weak convergence of measures
\begin{equation}\label{step5:main-eq}\sum_{\sigma} \mathbf{1}_{\sigma} \mu_{\sigma,k} \to \Leb_{\Delta}. \end{equation} 

As discussed in \S \ref{sec:comparsion}, we cannot directly go from (\ref{eq:step5-weak-conv-before-res}) to (\ref{step5:main-eq}), because some mass of $\mu_{\sigma,k}$ can be trapped outside of the closed set $\sigma$. To get around this issue, we want to find some nice approximation $v_k$ such that \begin{equation}\label{eq:step5-equiv-of-meas} \lim_{k \to \infty} \mu_{\sigma,k}(\mathbf{L}_\sigma \backslash \sigma) = 0.\end{equation} That is, the mass of $\mu_{\sigma,k}$ concentrates on $\sigma$ as $k \to \infty$. Assuming (\ref{eq:step5-equiv-of-meas})—to be proved in Proposition~\ref{prop:step5-no-extra-mass}—and using (\ref{eq:step5-weak-conv-before-res}), we obtain (\ref{step5:main-eq}).

\begin{construction}\label{cons-mollification}
    We now construct the desired approximations $v_k$ of $v$. We continue using the notation in Convention~\ref{convention:step4}, but assuming additionally that the cell $C \in \mathcal{C}_n$ contains $\sk(X)$. In particular, the functions $X_0, \cdots, X_{n-d-1}, X_{m+1}, \cdots, X_{m+d}$ form coordinates on $\mathbf{L}$.

    The trick is to convolve $v$ with a mollifier $\eta$, that is, a smooth non-negative function on $\mathbf{L}$ integrating to $1$, additionally required to be compactly supported in \[O \coloneqq \{X_0, \cdots, X_{n-d-1} > 0 \} \subseteq \mathbf{L}.\] For any $k \in \N$, set $\eta_k(x) \coloneqq k^n \eta(k \cdot x)$, whose support is also in $O$. For a convex function $f$, we use the notation $f_k$ to denote the mollification $f \ast \eta_k$. By general properties of mollification, each $f_k$ is smooth, convex, and locally uniformly converges to $f$ as $k \to \infty$. 
\end{construction}
\begin{proposition}\label{prop:step5-no-extra-mass}
    With the construction above, (\ref{eq:step5-equiv-of-meas}) holds. 
\end{proposition}
\begin{proof}
     By definition of the tropicalization map, $\sigma = \tau^\trop(C)$ is the octant in $\mathbf{L}$ where all coordinates are non-negative. For any $x \in \mathbf{L}$, we have \[(v|_{\mathbf{L}})_k(x) \coloneqq (v|_{\mathbf{L}} \ast \eta_k) (x) \coloneqq \int_{\mathbf{L}} (v|_{\mathbf{L}})(x - y) \cdot \eta_k(y)dy = \int_{O} (v|_{\mathbf{L}})(x - y) \cdot \eta_k (y) dy,\]
    where the last equality uses that $\eta_k$ is supported in $O$ by Construction~\ref{cons-mollification}. 

    Let $V$ be the closed subset $\{X_{m+1} \leq 0\} \cup \cdots \cup\{X_{m+d} \leq 0\} \backslash \mathrm{Int}(\sigma)$ of $\mathbf{L}$. By the Portmanteau theorem and Proposition~\ref{prop:step4-cases}, we have \[\limsup_{k \to \infty} \RMA((v|_{\mathbf{L}})_k)(V) \leq \RMA(v|_{\mathbf{L}})(V) = 0.\] 
    
    So consider $U \coloneqq \mathbf{L} \backslash (\sigma \cup V)$, which is open. Note all $X_{m+1}, \cdots, X_{m+d}$ are positive on $U$. To prove \eqref{eq:step5-equiv-of-meas} it then suffices to show that for all $k$ we have \[\RMA((v|_{\mathbf{L}})_k)(U) = 0.\]
 
   Let us define another convex function $w: \mathbf{L} \to \mathbf{R}$ by 
    \[w(x) = \max_{0 \leq i \leq n-d-1} \{u(X_{m+1} - X_i, \cdots, X_{m+d} - X_i)\}.\] Crucially, $w$ only depends on $n-1$ linear forms, namely \[X_{m+1} - X_0, \cdots, X_{m+d} - X_0, X_0 - X_1, \cdots, X_0 - X_{n-d-1}.\]
    
    For any $x \in U$ and any $y \in O$, we have \begin{align*}\begin{split} &v|_{\mathbf{L}}(x - y) \\ =& \max_{0 \leq i \leq n-d-1} \big\{u((X_{m+1} - Y_{m+1}) - (X_i - Y_i), \cdots, (X_{m+d} - Y_{m+d}) - (X_i - Y_i)),\\
    &u(X_{m+1} - Y_{m+1}, \cdots, X_{m+d} - Y_{m+d})\big\}\\
    =& \max_{0 \leq i \leq n-d-1} \{u((X_{m+1} - Y_{m+1}) - (X_i - Y_i), \cdots, (X_{m+d} - Y_{m+d}) - (X_i - Y_i))\}\\
    =& \ w(x-y),\end{split}\end{align*}
    where we used that $-(X_i - Y_i) = Y_i - X_i > 0$ for all $0 \leq i \leq n-d-1$. 
    
    So for any $x \in U$ and any $k \in \mathbf{N}$, the following two convolutions are equal, \[(v|_{\mathbf{L}})_k(x) = w_k(x).\] Because $U$ is open, we have \begin{align*}\RMA((v|_{\mathbf{L}})_k)(U) = |D^2((v|_{\mathbf{L}})_k)(U)| = |D^2(w_k)(U)| = \RMA(w_k)(U) = 0, \end{align*} using that $w$ depends on less than $n = \dim(\mathbf{L})$ variables. So $\RMA((v|_{\mathbf{L}})_k)(U) = 0$ as desired.
\end{proof}
All in all, we have established Theorem~\ref{intro:thm-1}; more precisely, 
\begin{theorem}\label{thm:A}
    The function \[\psi^\na(\nu) = u(\frac{e}{b_1} \nu(s_1), \cdots, \frac{e}{b_d}\nu(s_d))\] is a continuous plurisubharmonic function on $X^\na$, and solves the non-Archimedean Monge--Ampère equation \[\NAMA(\psi^\na) = \mu^\na\] on the Berkovich space $X^\na$, where $\mu^\na$ is the Lebesgue measure on the essential skeleton $\sk(X)$. 
\end{theorem}

\section{Interlude III: hybrid spaces}\label{sec:hybrid-spaces}

We review the construction and properties of hybrid spaces, as well as a measure convergence result on the hybrid space associated to a log Calabi--Yau pair.

\subsection{Construction of hybrid spaces}\label{hyb:setup}\hfill

We will work with the isotrivial hybrid space over $[0,1]$, studied in \cite{Ber09}, \cite{Jon16} and \cite{BJ17}. There is also a version of hybrid spaces over the unit disk that allows for non-isotrivial fibers in these works. 

For any $\rho \in [0,1]$, define a \emph{hybrid norm} on $\mathbf{C}$,
\[\|\cdot\|_{\mathrm{hyb}} \coloneq \max\{|\cdot|_0, |\cdot|_{\infty}\},\] where $|\cdot|_0$ is the trivial absolute value, and $|\cdot|_\infty$ is the usual Euclidean absolute value. The Berkovich spectrum $\mathcal{M}(\mathbf{C}, \|\cdot\|_{\mathrm{hyb}})$ is the set of all multiplicative seminorms on $\C$ bounded above by $\|\cdot\|_{\mathrm{hyb}}$. We define a map \[\rho\colon[0,1] \to \mathcal{M}(\mathbf{C}, \|\cdot\|_{\mathrm{hyb}}),\] given by $\tau \mapsto |\cdot|_{\infty}^\tau$. It is an exercise to see $\rho$ is an isomorphism. 

For any complex scheme of finite type $Y$, the hybrid space $Y^\hyb$ is the Berkovich analytification of $Y$ over $(\C, \|\cdot\|_{\mathrm{hyb}})$. We now briefly explain the construction of $Y^\hyb$ and refer to \cite[Appendix A]{BJ17} for details. For any affine chart $U \coloneqq \spec(A) \subseteq Y$, let $U^\hyb$ be the set of all multiplicative seminorms on $A$ \emph{bounded} above by $\|\cdot\|_{\mathrm{hyb}}$, equipped with the weakest topology making evaluation on any function $f \in A$ continuous. By gluing one obtains $Y^\hyb$. 

There is a continuous structural map \[\lambda\colon Y^\hyb \to \mathcal{M}(\mathbf{C}, \|\cdot\|_{\mathrm{hyb}}) \overset{\rho}{=} [0,1],\] which is open by \cite[Theorem C]{Jon16}, and in fact flat in a suitable sense; see \cite[Proposition 6.6.10]{LP24}. The zero fiber $\lambda^{-1}(0) \subseteq Y^\hyb$ is isomorphic to $Y^\na$, the Berkovich analytification of $Y$ over $(\mathbf{C}, \|\cdot \|_0)$. In addition, for any $\tau \in (0,1]$, the fiber $\lambda^{-1}(\tau)$ is homeomorphic to the complex analytification of $Y$, which we also denote by $Y$ for simplicity. Thus the name \emph{hybrid} for $Y^\hyb$. 

\begin{example}\label{eg:cont-fn-on-hybrid}
    By definition of topology on $Y^\hyb$, if $U \subseteq Y$ is Zariski open and $f$ is a regular function on $U$, the following function \[x \mapsto \begin{cases}
        |f(x)|^{\lambda(x)} & \lambda(x) \neq 0,\\
        |f(x)| & \lambda(x) = 0
    \end{cases}\]
    is continuous on $U^\hyb \subseteq X^\hyb$. 
\end{example}

\subsection{Hybrid continuity}\hfill

We briefly review the notion of continuous hybrid metrics on line bundles over $Y^\hyb$. This has been developed in \cite{Fav20} and \cite{PS23} in the setting of hybrid space over the unit disk, and works almost verbatim in the isotrivial setting over $[0,1]$. We now give an explicit description. 

Let $L$ be a line bundle over $Y$. It induces a line bundle $L^\hyb$ on $X^\hyb$, such that \[L^\hyb|_{\lambda^{-1}(0,1]} = L \times (0,1], \text{ and } L^\hyb|_{\lambda^{-1}(0)} = L^\na.\]

\begin{convention}\label{conv:tau-and-t}To agree with the scaling in the literature, we use the parameter $t$ on $\mathcal{M}(\mathbf{C}, \|\cdot\|_{\hyb})$, satisfying $t = e^{-1/\tau}$ and we set $t = 0$ when $\tau = 0$.\end{convention}

\begin{definition}(\cite[Definition 2.1]{PS23})
    A hybrid metric $\phi$ on $L^\hyb$ is a function on the complement of the zero section in the total space of $L^\hyb$, which respects restriction and satisfies $\phi \circ (fs) = |f| \cdot (\phi \circ s)$ for any local section $s$ of $L$ and local function $f$. A hybrid metric is continuous if it is continuous as a function.
\end{definition}

Just like in the complex case, if a Zariski open set $U \subseteq X$ trivializes $L$, say by some local section $s$, then we can locally identify the metric $\phi$ on $L^\hyb$ with a function $\psi$ on $U^\hyb$, such that for any other local section $s'$ of $L$ over $U$, we have \[\phi \circ s' =|s'/s| \cdot e^{-\psi},\]
which makes sense because $s'/s$ is a regular function and so we can take its absolute value. This allows us to characterize continuous hybrid metrics as follows. 

\begin{proposition}\label{prop:equiv-hybrid-func}
    A hybrid metric $\phi$ is continuous on $L^\na$ if and only if the function $\psi$ is continuous on $U^\hyb$ for any Zariski open set $U \subseteq X$ that trivializes $L$.
\end{proposition}

\begin{example}\label{eg:FS-hybrid-metric}(\cite[Proposition 3.5]{PS23})
    Suppose for some $m > 0$ there are global sections $\{s_0, \cdots, s_N\} \in H^0(Y, mL)$ without common zeros. A fundamental example of a continuous hybrid metric is the hybrid Fubini--Study metric given by \[\phi_t = \begin{cases}
        (2m\log{t^{-1}})^{-1}\log(\sum_{0 \leq i \leq N} |s_i|^2) & t \neq 0 \\[3pt]
        m^{-1}\max_{0 \leq i \leq N} (\log|s_i|) & t = 0.
    \end{cases}\] Note that $\phi_0$ is a non-Archimedean Fubini--Study metric, as in Definition~\ref{def:triv-metric-over-triv-valued-field}. 
\end{example}

\subsection{Measure convergence on hybrid spaces}\label{sec:meas-convergence}\hfill

We return to the smooth log Calabi--Yau pair $(\bar{X}, D)$ from Setup \ref{setup}, where $\bar{X}$ is a smooth projective Fano variety of dimension $n$ and $D$ is a reduced SNC anticanonical divisor whose dual complex is a $d$-dimensional simplex with $d < n$. Recall that $X = \bar{X} \backslash D$ has a non-vanishing regular $n$-form $\Omega \in H^0(X, K_{\bar{X}}|_X)$ with a simple pole along $D$. For any $t \neq 0$, we equip the fiber $\lambda^{-1}(t) \subseteq X^\hyb$ with a smooth positive measure (also a volume form) \[\mu_t \coloneqq \frac{(\sqrt{-1})^{n^2}\Omega \wedge \bar{\Omega}}{(2\pi \log t^{-1})^d}.\] Recall also that $\mu^\na$ is the Lebesgue measure on the essential skeleton $\sk(X) \subseteq X^\na = \lambda^{-1}(0).$ The following convergence theorem connects these measures, and can be deduced from \cite[Theorem A]{Shi22}. See also \cite{AN25} for a stronger version of this theorem. 

\begin{theorem}\label{thm:Shivaprasad}
    On $X^\hyb$, the measures $\mu_t$ converge weakly to the measure $\mu^\na$. 
\end{theorem} This justifies why the measure $\mu^\na$ is a natural non-Archimedean replacement for the volume form $\mu^{\mathrm{a}} = (\sqrt{-1})^{n^2}\Omega \wedge \bar{\Omega}$. 

\section{Hybrid continuity}\label{sec:hybrid-continuity}

In this section we prove Theorem~\ref{intro:thm-2}, recalling its setup briefly first. Let $(\bar{X}, D)$ be any of the following pairs, 

\begin{enumerate}
    \item(Tian--Yau) $\bar{X}$ is a smooth projective Fano variety of dimension $n$, and $D$ is an irreducible smooth anticanonical divisor; 
    \item(Collins--Li) $\bar{X}$ is a smooth projective Fano variety of dimension $n$, and $D$ is an anticanonical divisor with irreducible components $D_1, D_2$ such that $D_1 \cap D_2$ is irreducible; 
    \item(Calabi model space) let $Y$ be a smooth projective Fano variety of dimension $n$, with a reduced SNC anticanonical divisor $D$ whose deepest intersection stratum $Z$ is irreducible. Then $\bar{X}$ is the relative normal bundle $N_{Z/Y}$ and $D$ is identified with the zero section of $\bar{X}$. 
\end{enumerate}

In all cases above, we set $X = \bar{X} \backslash D$. By \cite{TY90} and \cite{CL24}, in the first two cases $X$ admits a complete Archimedean Calabi--Yau potential $\psi^{\mathrm{a}}$. In the third case, by \cite{CTY24} there is an Archimedean Calabi--Yau potential near $D$, and we extend it to any smooth potential $\psi^{\mathrm{a}}$ on $X$ bounded away from $D$. The choice of extension does not matter. By Theorem~\ref{intro:thm-1}, in each case above, the associated Berkovich analytification $X^\na$ likewise carries a non-Archimedean complete Calabi--Yau potential $\psi^\na$. We will amalgamate the two potentials $\phi^{\mathrm{a}}, \phi^\na$ into a function on the hybrid space $X^\hyb$ and show it is continuous. This hybrid continuity suggests that non-Archimedean Calabi--Yau potentials should be used to prescribe and capture limiting behavior of their Archimedean counterparts.

In the projective and non-isotrivial setting, such hybrid continuity has been established in \cite{Li25b} and \cite{Li25a}. For the case of abelian varieties this was achieved in \cite{GO24}. Our proof amounts to using the fact that the complete Archimedean Calabi--Yau potentials are governed by the generalized Calabi ansatz, and then finding an appropriate scaling. 

\subsection{The Tian--Yau potential}\label{Conv:TY-metric}\hfill

Let $(\bar{X}, D)$ be a pair where $\bar{X}$ is a smooth projective Fano variety of dimension $n$, and $D$ is a smooth irreducible anticanonical divisor, cut out by a section $s \in H^0(\bar{X}, -K_{\bar{X}})$. Then $X = \bar{X} \backslash D$ is an affine Calabi--Yau variety, which has a non-vanishing regular $n$-form $\Omega$ with a simple pole along $D$. Fix a large integer $e$ such that $M \coloneqq - eK_{\bar{X}}$ is very ample and $\dim H^0(\bar{X},M) \geq n + 1$. Because $s^e$ is a global non-vanishing section of $M|_X$, we can identify a (continuous, smooth) metric $\phi$ on $M|_X$ with a (continuous, smooth) function $\phi - e\log|s|$ on $X$. 

In addition, choose a finite subset $\mathcal{S}' \coloneqq \{s_0', \cdots, s_m'\} \subseteq H^0(\bar{X}, M)$ with $m \geq n$ as in Step \hyperref[(1a)]{(1a)}, that is, 
\begin{itemize}
    \item $\mathcal{S} \coloneqq \mathcal{S}' \cup \{s^e\}$ is a basis of $H^0(\bar{X}, M)$;
    \item the intersection of the zero loci of any $k$ members of $\mathcal{S}$ is smooth and connected for $k \in [1,n-1]$, is reduced and of dimension $0$ when $k = n$, and is empty when $k \geq n+1$. 
\end{itemize}
In particular, the sections in $\mathcal{S}'$ have no common zero. 

\begin{construction}[\cite{TY90}]\hfill 

We now review the construction of the Tian--Yau potential, which is a smooth, strictly plurisubharmonic function $\psi_\TY$ on $X$ solving the complex MA equation \[(dd^c \psi_\TY)^n = \mathrm{const} \cdot \Omega \wedge \bar{\Omega}.\] This potential $\psi_\TY$ decomposes as \[\psi_\TY = (\phi_{\FS}^{\mathrm{a}} - e\log|s|+ \psi_{\mathrm{CY}} )^{(n+1)/n} + \psi_{\mathrm{pos}} + \theta,\] where,
\begin{enumerate}[label=(a\arabic*), ref=a\arabic*]
    \item\label{item:a1} $\phi_{\FS}^{\mathrm{a}}$ is the Archimedean Fubini--Study metric on $M$ associated to the sections $\{s_0', \cdots, s_m'\}$, and $\phi_{\FS}^{\mathrm{a}} - e\log|s|$ is the corresponding function on $X$;
    \item\label{item:a2} $\psi_{\CY}$ is a smooth function on $\bar{X}$ whose restriction to $D$ satisfies \[\omega_D + dd^c(\psi_\CY|_D) > 0 \text{ and }\mathrm{Ric}(\omega_D + dd^c (\psi_\CY|_D)) = 0,\] where $\omega_D = dd^c(\phi_{\FS}^{\mathrm{a}}|_D) = 0$ is a reference K\"{a}hler form on $D$; 
    \item\label{item:a3}  $\psi_{\mathrm{pos}}$ is a compactly supported smooth positive function on $X$ that makes the reference potential \[\psi_{\mathrm{ref}} \coloneqq (\phi_{\FS}^{\mathrm{a}} - e \log|s| + \psi_{\CY})^{(n+1)/n} + \psi_{\mathrm{pos}}\] strictly plurisubharmonic on $X$;
    \item\label{item:a4} $\theta$ is a smooth function on $X$ solving the equation \[(dd^c(\psi_\mathrm{ref} + \theta))^n = \mathrm{const} \cdot \Omega \wedge \bar{\Omega},\] and $\theta$ is bounded by \cite{Hei12}.
\end{enumerate}
\end{construction}

For any $t \neq 0$, we equip $\lambda^{-1}(t) \simeq X$ with a rescaled Tian--Yau potential,
\[\psi_t \coloneqq \frac{\psi_\TY}{(\log{t}^{-1})^{1/n}}.\]
The non-standard scaling $(\log{t}^{-1})^{1/n}$ is motivated by the measure convergence result (Theorem~\ref{thm:Shivaprasad}), and will be further discussed in Remark~\ref{rem:scaling}. As we will see in Proposition~\ref{prop:TY-hybrid}, this is the correct scaling in order to obtain hybrid continuity.

In addition, we equip the central fiber $\lambda^{-1}(0) \simeq X^\na$ with the potential $\psi^\na$, the solution to the equation (\ref{eq:NA-MA}) from Theorem~\ref{intro:thm-1}. Since $D$ is irreducible, the discussion in \S \ref{sec:generalized-calabi-ansatz} for the case when $d = 1$ allows us to write down $\psi^\na$ very explicitly. Indeed, define a non-Archimedean Fubini--Study metric on $M^\na$ associated to the chosen sections $\mathcal{S}'$: \[\phi_{\FS}^\na \coloneqq \max_{0 \leq j \leq m} \{\log|s_j'|\},\] which equals the trivial metric on $M^\na$ by Remark \ref{def:triv-metric-over-triv-valued-field}. The solution to the NA MA equation $\NAMA(-) = \mu^\na$ is then simply \[\psi^\na(\nu) = \nu(s^e)^{(n+1)/n} = (\phi_{\FS}^{\na} - e\log|s|)^{(n+1)/n},\]where the notation $\nu(s^e)$ is as in Definition~\ref{def:triv-metric-over-triv-valued-field}. 

\begin{proposition}\label{prop:TY-hybrid}
    The function $\psi^\hyb \coloneqq (\psi_t, \psi_0 = \psi_\na)$ is continuous on $X^\hyb$.
\end{proposition}

\begin{proof}
    Consider any net $p_k$ that converges to some point $\nu \in \lambda^{-1}(0)$. By passing to a subnet, we can assume $p_k$ lies in some Archimedean fiber $\lambda^{-1}(t_k)$ where $t_k \downarrow 0$ as $k \to \infty$.  We must show \[\frac{\psi_{t_k}(p_k)}{\log{t_k^{-1}}} \to \psi_0(\nu).\]
    By Example~\ref{eg:FS-hybrid-metric} we have \[\lim_{k \to \infty}(\phi_{\FS}^{\mathrm{a}} - e\log|s|)(p_k) = (\phi_{\FS}^\na - e\log|s|)(\nu).\]
    Now, because the functions $\psi_{\mathrm{pos}}$, $\psi_\CY$, and $\theta$ are bounded, we have \begin{align*}\lim_{k \to \infty} \frac{\psi_{t_k} (p_k)}{\log{t_k^{-1}}} =& \lim_{k \to \infty} \frac{\psi_\TY(p_k)}{(\log{t_k^{-1}})^{(n+1)/n}} = \lim_{k \to \infty} \Big(\frac{\phi_{\FS}^{\mathrm{a}} - e\log|s|}{\log{t_k^{-1}}}\Big)^{(n+1)/n}(p_k)\\ =& (\phi_{\FS}^\na - e\log|s|) ^{(n+1)/n}(\nu) = \psi_0(\nu),\end{align*}completing the proof. 
\end{proof}

\subsection{The Collins--Li potential}\label{Conv:CL-metric}\hfill

We restate Setup \ref{setup} when $d = 2$. Let $(\bar{X}, D)$ be a pair where $\bar{X}$ is a smooth projective Fano variety, with $D \coloneqq D_1 + D_2$ a reduced simple normal crossing anticanonical divisor. For $i = 1, 2$, let $s_i \in H^0(\bar{X}, b_iL)$ be the defining section of $D_i$. Recall that $(b_1 + b_2) L = -K_{\bar{X}}$. Fix some positive integer $e$ divisible by $b_1, b_2, b_1 + b_2$ such that $eL$ is very ample and $\dim H^0(\bar{X}, eL) \geq n + 2$. Choose sections $\mathcal{S}' \coloneqq \{s_0', \cdots, s_m'\} \subseteq H^0(\bar{X}, eL)$ with $m \geq n$ as in Step \hyperref[(1a)]{(1a)}, that is,
\begin{itemize}
    \item $\mathcal{S} \coloneqq \mathcal{S}' \cup \{s_1^{e/b_1}, s_2^{e/b_2}\}$ is a basis of $H^0(\bar{X}, eL)$;
    \item the intersection of the zero loci of any $k$ members of $\mathcal{S}$ is smooth and connected for $k \in [1,n-1]$, is reduced and of dimension $0$ when $k = n$, and is empty when $k \geq n+1$.
\end{itemize}

In particular, the sections in $\mathcal{S}'$ have no common zero. 

\begin{construction}[\cite{CL24}]\hfill

    Let $\phi_{\FS}^{\mathrm{a}}$ be the Fubini--Study metric on $(\bar{X}, eL)$ associated to $\mathcal{S}'$.
    For $i = 1, 2$, define a function on $X$, \[z_i = \phi_{\FS}^{\mathrm{a}} - \frac{e}{b_i} \log|s_i|.\] Note that $s_{2}^{b_1} / s_{1}^{b_2}$ is a regular function on $X$, so one can also define a pluriharmonic function $\widetilde{z_{1}} = \frac{1}{b_2} \log|s_{2}^{b_1} / s_{1}^{b_2}|$ on $X$. Similarly define $\widetilde{z_2}$. 

    The Collins--Li potential $\psi_\CL$ is a smooth, strictly plurisubharmonic function on $X$ such that \[(dd^c \psi_\CL)^n = \mathrm{const} \cdot \Omega \wedge \bar{\Omega},\] where $\Omega$ is a non-vanishing regular $n$-form on $X$ with a simple pole along $D$. This potential can be decomposed into three parts, \[\psi_\CL = \psi_{\infty} + \psi_{\mathrm{pos}} + \theta,\] where, 
    \begin{enumerate}[label=(b\arabic*), ref=b\arabic*]
        \item\label{item:b1} $\psi_{\infty}$ is a smooth extension of a function $\psi_{\infty, 0}$ defined on a punctured neighborhood of $D$ in $X$, that is, the region where $1 \ll \max(z_1, z_2)$. The local $(k, \alpha)$-norm of the difference between $\psi_{\infty, 0}$ and the generalized Calabi ansatz $u$ from Theorem~\ref{thm:CTY} is \[\|\psi_{\infty, 0} - u\|_{k, \alpha, \mathrm{loc}} = O(z_1^{-\frac{2n-1}{n-1}}) = O(z_2^{-\frac{2n-1}{n-1}});\]
    \item\label{item:b2}$\psi_{\mathrm{pos}}$ is a compactly supported function on $X$ such that $\psi_\infty + \psi_{\mathrm{pos}}$ is strictly plurisubharmonic;
    \item\label{item:b3} $\theta$ is the output of the Tian--Yau--Hein package; it is a smooth function on $X$ with local $(k, \alpha)$-norm $\|\theta\|_{k,\alpha, \mathrm{loc}} = O((z_1^2 + z_2^2 + 1)^{-\frac{q(n+2)}{4n}})$ for any $q \in (0, \frac{2n-4}{n+2})$ and any $k$. 
    \end{enumerate} 
\end{construction}

Similar to \S \ref{Conv:TY-metric}, for $t \neq 0$, we equip $\lambda^{-1}(t) \simeq X$ with a rescaled Collins--Li potential, where the scaling is again motivated by Theorem \ref{thm:Shivaprasad}, 
\[\psi_t \coloneqq \frac{\psi_\mathrm{CL}}{(\log{t^{-1}})^{2/n}}.\] 

On the other hand, by Example~\ref{eg:FS-hybrid-metric}, for $i = 1, 2$ the function \[z_i \cdot (\log{t^{-1}})^{-1}\text{ on }\lambda^{-1}(\{t \neq 0\}) \subseteq X^\hyb\] extends continuously to all of $X^\hyb$, where on $\lambda^{-1}(0)$ it equals \[\max_{0 \leq j \leq m} \log|s_j'| - \frac{e}{b_i}\log|s_i|.\] We denote this extension or its restriction to any fiber as $z_i$. Then, as obtained in Theorem \ref{intro:thm-1}, the central fiber $\lambda^{-1}(0) \simeq X^\na$ has a potential that solves (\ref{eq:NA-MA}), namely \[\psi^\na = u(z_1, z_2).\]

\begin{proposition}\label{prop:CL-hybrid}
    The function $\psi^\hyb \coloneqq (\psi_t, \psi_0 = \psi^\na)$ is continuous on $X^\hyb$. 
\end{proposition}

\begin{proof}
    Consider any net $p_k$ that converges to some point $\nu \in \lambda^{-1}(0)$. By passing to a subnet, we can assume $p_k$ lies in some Archimedean fiber $\lambda^{-1}(t_k)$ where $t_k \downarrow 0$ as $k \to \infty$.  We must show \begin{equation}\label{eq:CL-hyb-conv}\frac{\psi_{t_k}(p_k)}{\log{t_k^{-1}}} \to \psi_0(\nu).\end{equation}
    
    It follows from the estimates of $U_1, \psi_{\mathrm{pos}}$ and $\theta$ in $x_1, x_2$ from (\ref{item:b1})--(\ref{item:b3}) that, as $k \to \infty$, 
    \begin{align*}\Big|\frac{\psi_{t_k}(p_k)}{\log{t_k^{-1}}} - \frac{\psi_{\infty}(p_k)}{(\log{t_k^{-1}})^{(n+2)/n}}\Big| = \Big| \frac{\psi_\CL(p_k)}{(\log{t_k^{-1}})^{(n+2)/n}} - \frac{\psi_{\infty}(p_k)}{(\log{t_k^{-1}})^{(n+2)/n}}\Big| \to 0.\end{align*}
    
    So we focus on the function $\psi_{\infty}$, and proceed by discussing two cases depending on the center $\mathbf{p} \coloneqq c_{\bar{X}}(\nu) \in \bar{X}$.

    \noindent \emph{$\bullet$ Case 1: $\mathbf{p} \in X$.}
    
    In this case the sequence $z_i(p_k) = o(\log{t_k^{-1}})$ as $k \to \infty$.\footnote{In the language of \cite{Thu07}, the Berkovich space $X^\na$ decomposes as two parts $X^{\beth} \sqcup X^\infty$, and Case 1 exactly means that $\nu$ is in $X^{\beth}$, and $\mathbf{p} \in X^\infty$ in the other case.} 
    By the estimate of $\psi_{\infty,0} - u$ in (\ref{item:b1}) and continuity of $u$, we have \[\lim_{k \to \infty} \frac{\psi_{\infty}(p_k)}{(\log{t_k^{-1}})^{(n+2)/n}} = 0.\] On the other hand, the non-Archimedean potential satisfies
    \[\psi_0(\nu) = u(0,0) = 0\] by homogeneity of $u$. So (\ref{eq:CL-hyb-conv}) is verified. 
    
    \noindent \emph{$\bullet$ Case 2: $\mathbf{p} \not\in X$.} This means $1 \ll \max\{z_1(p_k), z_2(p_k)\}$ as $k \to \infty$. Again, the potential $\psi_{\CL}$ is governed by $\psi_{\infty}$, which is close to $u(z_1, z_2)$ in this region by (\ref{item:b1}). Because $u$ is homogeneous, we have \[\frac{u(z_1(p_k), z_2(p_k))}{(\log{t_k^{-1}})^{(n+2) / n}} = u(\frac{z_1(p_k)}{\log{t_k^{-1}}}, \frac{z_2(p_k)}{\log{t_k^{-1}}}),\] which, by continuity of $u$, as $k \to \infty$ approaches \[u(z_1(\nu),z_2(\nu)) = \psi_0(\nu),\] completing the proof. 
\end{proof}

\subsection{The generalized Calabi Ansatz}\label{Conv:CA}\hfill

This situation is similar to the previous two. Recall Setup \ref{setup}: $Y$ is a smooth projective Fano variety of dimension $n$, $D$ is a reduced SNC anticanonical divisor with components $D_1, \cdots, D_d$ for some $d < n$, cut out by sections $s_i \in H^0(Y, b_iL)$ for some ample ($\mathbf{Q}$-)line bundle $L$, and $Z = D_1 \cap \cdots \cap D_d$ is a smooth connected projective Calabi--Yau variety. 

The bundle $N_{Z / Y}$ admits a direct sum decomposition $b_1 L|_Z \oplus \cdots \oplus b_d L|_Z$. Fix some positive integer $e$ divisible by $b_1, \cdots, b_d, b_1 + \cdots +b_d$ such that $eL$ is very ample and $\dim H^0(\bar{X}, eL) \geq n + d$. We choose sections $\mathcal{S}' \coloneqq \{s_0', \cdots, s_m'\} \subseteq H^0(Y, eL)$ with $m \geq n$ as in Step \hyperref[(1a)]{(1a)}, that is,
\begin{itemize}
    \item $\mathcal{S} \coloneqq \mathcal{S}' \cup \{s_1^{e/b_1}, \cdots, s_d^{e/b_d}\}$ is a basis of $H^0(Y, eL)$;
    \item the intersection of the zero loci of any $k$ members of $\mathcal{S}$ is smooth and connected for $k \in [1,n-1]$, is reduced and of dimension $0$ when $k = n$, and is empty when $k \geq n+1$.
\end{itemize}

Let $N^{\times}$ denote the complement of the zero section in $N_{Z / Y}$, and let $\phi_\FS^{\mathrm{a}}$ be the Archimedean Fubini--Study metric on $L|_Z$ associated to the sections in $\mathcal{S}'$. Following the construction in \S \ref{sec:calabi-model-space}, let $-\log h \coloneqq \phi_{\FS}^{\mathrm{a}} + \psi_{\CY}$ be the potential of the Ricci flat metric on $L|_Z$, where $\psi_\CY$ is the Calabi--Yau potential. By taking multiples, we have a metric $x_i = -\log h_i$ on $b_iL|_Z$ for each $1 \leq i \leq d$. Let $u$ be the solution to the PDE in Theorem~\ref{thm:CTY}. Then the function \[u(x_1, \cdots, x_d)\] is the potential of a Ricci flat metric on a neighborhood of the zero section in $N^{\times}$. We extend $u(x_1, \cdots, x_d)$ to any smooth function $\psi^{\mathrm{a}}$ on $N^{\times}$ bounded away from $D$. The choice of extension does not matter. 

We will work on the hybrid space $(N^\times)^\hyb$. First, on a general fiber $\lambda^{-1}(t)$ for $t \neq 0$, we put a rescaled metric \[\psi_t \coloneqq \frac{\psi^{\mathrm{a}}}{(\log{t^{-1}})^{d/n}}.\]

The central fiber $\lambda^{-1}(0) = (N^\times)^\na$ has essential skeleton $\sk(N^\times) = \QM(N^\times, \sum_{i=1}^d D_i)$. For $1 \leq i \leq d$, define a function \[z_i \coloneq \max_{0 \leq j \leq m} (\log|s_j'|) - \frac{e}{b_i}\log|s_i| \text{ on } (N^\times)^\na.\]By Theorem~\ref{intro:thm-1}, the function \[\psi^\na = u(z_1, \cdots, z_d)  \text{ on } (N^\times)^\na\] solves the (\ref{eq:NA-MA}). The following proposition follows from a similar argument as in Propositions~\ref{prop:TY-hybrid} and~\ref{prop:CL-hybrid}, though the proof is simpler here because there are fewer correction terms.

\begin{proposition}\label{prop:CA-hybrid}
    The function $\psi^\hyb \coloneqq (\psi_t, \psi_0 = \psi^\na)$ is continuous on $(N^\times)^\hyb$.
\end{proposition}

Combining Propositions~\ref{prop:TY-hybrid}, \ref{prop:CL-hybrid}, and \ref{prop:CA-hybrid}, we have shown Theorem~\ref{intro:thm-2}. 

\section{Concluding remarks}\label{sec:concluding}

We conclude with two remarks on recovering the homogeneous degree of Calabi ansatz from measure convergence, and on a volume growth conjecture by Odaka.

\subsection{Recovering the homogeneous degree}\label{rem:scaling}\hfill

In general, convergence of measures does not imply convergence of the corresponding potentials. However, we can recover the homogeneous degree $\alpha$ of the generalized Calabi ansatz $u$ from the hybrid convergence of measures, which we now explain. This observation does not assume anything a priori about the continuous hybrid metric beyond homogeneity, and provides evidence for the canonicity of the generalized Calabi ansatz.
    
We keep the same setting and notations as in \S \ref{Conv:CA}. 
Say $v$ is any $C^2$ function in $d$ variables, homogeneous of degree $\alpha > 0$. Using Proposition~\ref{prop:equiv-hybrid-func}, Example~\ref{eg:FS-hybrid-metric} and homogeneity of $v$, the following hybrid function is continuous on $(N^\times)^\hyb$, \[\psi^\hyb \coloneqq \begin{cases}
    \psi_t \coloneqq v(\psi_1^{\mathrm{a}}, \cdots, \psi_d^{\mathrm{a}}) \cdot (\log{t^{-1}})^{1 - \alpha} & t \neq 0    \\[3pt]
    \psi_0 \coloneqq v(\psi_1^\na, \cdots, \psi_d^\na) & t =0.
\end{cases}\]
    
For convenience set also \[\psi \coloneqq v(\psi_1^\mathrm{a}, \cdots, \psi_d^{\mathrm{a}}).\] We assume now that $\psi^\hyb$ solves the complex or non-Archimedean Monge--Ampère equation fiberwise on $(N^\times)^\hyb$. To be precise, this means that for any $t \neq 0$, the complex Monge--Ampère measure \begin{equation}\label{eq:meas-conv-normal-bundle}\MA_{\mathbf{C}}(\psi_t) = \frac{\CMA(\psi)}{(\log{t^{-1}})^{n(\alpha-1)}}\end{equation} is the volume form on some neighborhood of the zero section in $N^\times$, and $\NAMA(\psi_0)$ is the Lebesgue measure on $\sk(N^\times) \subseteq \lambda^{-1}(0) = (N^\times)^\na$.
    
By Theorem~\ref{thm:Shivaprasad}, up to a multiplicative constant, there is weak convergence of measures \begin{equation}\label{eq:meas-convergence-d}\frac{\MA_{\mathbf{C}}(\psi)}{(\log{t^{-1}})^d} \to \NAMA(\psi_0).\end{equation} Comparing (\ref{eq:meas-conv-normal-bundle}) with (\ref{eq:meas-convergence-d}), we see that if the fiberwise MA measure of the continuous function $\psi^\mathrm{hyb}$ convergences, then $n(\alpha-1) = d$, or $\alpha = \frac{n+d}{n}$. This is exactly the homogeneous degree obtained from the dimensional analysis of the generalized Calabi ansatz in \cite[\S 2.7]{CL24}. 

\subsection{A conjecture of Odaka}\label{rem:Odaka-conj}\hfill

Last but not least, we explain a conjecture of Odaka, which says that, if an open Calabi--Yau manifold admits a complete Ricci flat metric, then the volume growth of geodesic balls therein is bounded below by the dimension of the essential skeleton in the Berkovich analytification; see \cite{Oda20}. An elementary calculation as follows shows that complete Calabi--Yau potentials $\psi$ which are sufficiently asymptotic to the generalized Calabi ansatz $u$, if they exist, satisfy Odaka's conjecture.

The homogeneous degree of $u$ being $(n+d) / n$ translates into \[u \in O(|r|^{\frac{n+d}{n}}),\] where $r = (r_1, \cdots, r_d)$ are the log radius on the split normal bundle $N_{Z/\bar{X}} = b_1L|_Z \oplus \cdots \oplus b_d L|_Z$. Then the distance to the zero section is of order \smash{$O(|r|^{\frac{n+d}{2n}})$}. See \cite[\S 2.7]{CL24} for a calculation. The volume growth dimension $dd^c \psi$, viewed as a Riemannian metric, is \[\mathrm{vd}(dd^c\psi) = \lim_{r \to \infty} \log_r (r^{\frac{2nd}{n+d}}) = (2nd)/(n+d).\] Recall that the essential skeleton of the log pair $(\bar{X},D)$ has dimension $d$ when $D$ has $d$ irreducible components. The inequality in Odaka's conjecture therefore translates into $d \leq (2nd) / (n+d)$ which is equivalent to $d \leq n.$ This is true by the standing assumption made in \S \ref{sec:generalized-calabi-ansatz} to apply the Lefschetz hyperplane theorem. 

\section*{Appendices}
\renewcommand{\thesubsection}{\Alph{subsection}}
\renewcommand{\thefigure}{\thesubsection.\arabic{figure}}
\setcounter{figure}{0}

\renewcommand{\thetheorem}{\thesubsection.\arabic{theorem}}
\setcounter{theorem}{0}

\subsection*{A. Extending the CTY solution}\label{appen:A}\hfill

We prove Proposition~\ref{prop:extension}, extending the solution $u$ in Theorem~\ref{thm:CTY} from $\mathbf{R}^d_{\geq 0}$ to $\mathbf{R}^d$, using a general $C^1$ extension result for convex functions from \cite{AM17}.

\begin{proof}[Proof of Proposition~\ref{prop:extension}]
    Say $t_1, \cdots, t_d$ are the coordinates of $\R^d$. Let $\Sigma$ denote the regular $d$-simplex with vertices $(1, 0, \cdots, 0), \cdots, (0, \cdots, 0, 1).$  We identify the positive octant $\R^d_{\geq 0}$ as the cone over $\Sigma$ with cone direction $(1, \cdots, 1)$. From the construction of $u$ in \cite[\S 5]{CTY24}, for any $t \in \R^d_{\geq 0}$, we have \[u(t) = (\|t\| \cdot u(t/\|t\|))^{(n+d)/n}.\]

By \cite[\S 3]{CTY24}, the restriction of $u$ to $\Sigma$ is positive, convex, and differentiable. We denote this restriction by $\bar{u}$. First extend $u$ to some positive and differentiable function on the affine space $\mathbf{L}_\Sigma \simeq \R^{d-1}$, which can always be achieved but the resulting function is not necessarily convex. But by the proof of \cite[Theorem 1.10]{AM17} there is a strictly positive correction function $\psi$ such that the lower convex envelop $\psi_1$ of of $\bar{u} + \psi$ is still continuously differentiable on $\mathbf{L}_\Sigma$. Since $\bar{u}, \psi > 0$, the constant function $\inf\{\bar{u} + \psi\}$ is strictly positive. This constant function participates in the lower convex envelope defining $\psi_1$, so $\psi_1 > 0$ as well. 

Now we extend $\psi_1$ to a function $u_1$ defined on the open half-plane $\{t_1 + \cdots + t_d > 0\}$ by \[u_1(t) \coloneqq (\|t\| \cdot \psi_1(t/\|t\|))^{(n+d)/n}.\] Since $\psi_1$ is convex, as its perspective function $u_1$ is convex too; see e.g. \cite[\S 3.2.6]{BV04}. 

Observe that $u_1$ extends continuously by zero to $\R^d$, which we denote by $u_2$. We now check that $u_2$ is convex as well. This is equivalent to showing its graph is convex, and follows from the fact that $u_1 > 0$ on $\{t_1 + \cdots + t_d > 0\}.$ 

Lastly, we check $u_2$ is continuously differentiable. By \cite[Theorem 25.2]{Roc97}, to show a convex function is differentiable, it suffices to check its two sided derivatives are finite and agree in $d$ linearly independent directions. Since $\psi_1$ is continuously differentiable on $\mathbf{L}_\Sigma$, this gives us well-defined derivatives in $d-1$ directions. Moreover, $u_2$ is differentiable in the positive diagonal direction $(1, \cdots, 1)$, thanks to the exponent $(n+d)/n$ being greater than $1$. As $u_2$ is finite on $\R^d$, by \cite[Theorem 25.5]{Roc97} it is continuously differentiable. So $u_2$ is our desired extension.
\end{proof}

\setcounter{theorem}{0}  

\bibliographystyle{alpha}
\bibliography{references}
\nocite{CG25}\nocite{Che24}\nocite{HSVZ22}\nocite{Duc12}\nocite{Gut16}\nocite{Gut16}\nocite{Jel16}\nocite{NXY19}\nocite{KK10}\nocite{Bou25}\nocite{Ber90}\nocite{Gub16}

\end{document}